\documentclass[a4paper,11pt,reqno]{amsart}
\usepackage{graphicx}
\usepackage{a4wide}
\usepackage{amssymb}
\usepackage{amsthm}
\usepackage{eucal}
\usepackage[pdftex,colorlinks]{hyperref}

\numberwithin{equation}{section}

\newcommand{\be}{\begin{equation}}
\newcommand{\ee}{\end{equation}}
\newcommand{\ba}{\begin{eqnarray}}
\newcommand{\ea}{\end{eqnarray}}

\newtheorem{theorem}{Theorem}[section]
\newtheorem{proposition}[theorem]{Proposition}
\newtheorem{remark}[theorem]{Remark}
\newtheorem{lemma}[theorem]{Lemma}

\newtheorem{claim}[theorem]{Claim}

\begin{document}

\title [Controllability for  a Chemotaxis-Fluid model]{A controllability Result for a Chemotaxis-Fluid Model}

\author[F. W. Chaves-Silva]{F. W. Chaves-Silva$^{1,*}$ }
\thanks{$^1$Universit\'e de Nice Sophia-Antipolis, Laboratoire Jean A. Dieudonn\'e, UMR CNRS 6621, Parc Valrose, 06108 Nice Cedex 02,
France (\textit{fchaves@unice.fr}).  \\  \null $^2$Sorbonne Universit\'e, UPMC Univ. Paris 6, UMR 7598 Laboratoire Jacques-Louis Lions, Paris, F-75005 France (\textit{guerrero@ann.jussieu.fr}).
\\  $^*$F. W. Chaves-Silva has been supported  by  the ERC project Semi Classical Analysis of  Partial Differential Equations, ERC-2012-ADG, project number 320845}
\author[S. Guerrero]{S. Guerrero$^2$}




\maketitle

\begin{abstract}
In this paper we study the controllability of a coupled Keller-Segel-Navier-Stokes system. We show the local exact controllability of the system around some particular trajectories. The proof relies on new Carleman inequalities for the chemotaxis part and some improved Carleman inequalities for the Stokes system.

\vspace{0.5cm}

\noindent\textsc{R\'esum\'e.} Dans cet article, nous \'etudions la contr\^olabilit\'e  d'un syst\`eme de  Keller-Segel-Navier-Stokes coupl\'e. Nous montrons la contr\^olabilit\'e exacte locale du syst\`eme autour de quelques trajectoires particuli\`eres. La preuve repose sur de nouvelles in\'egalit\'es de Carleman pour la partie de la chimiotaxie et sur des in\'egalit\'es de Carleman am\'elior\'ees pour le syst\`eme de Stokes.
\end{abstract}


\section{Introduction and main results}
Let  $\Omega \subset \mathbb{R}^N$ (N = 2, 3) be a bounded connected open set whose boundary $\partial \Omega$ is regular enough. Let $T > 0$ and  $\omega_1$ and $\omega_2$  be two (small) nonempty subsets of $\Omega$, with $\omega_1 \cap \omega_2 \neq \emptyset$  when $N=3$. We will use the notation $Q =  \Omega \times (0,T) $ and  $\Sigma = \partial \Omega \times (0,T)$ and we will denote by $\nu(x)$ the outward normal to $\Omega$ at the point $x \in \partial \Omega$.

We introduce the  following usual spaces in  the context of fluid mechanics
$$
\bold{V} =   \{ u \in H^1_0(\Omega)^N; \ div \ u = 0\},
$$
$$
\bold{H} = \{ u \in L^2(\Omega)^N; \ div \ u = 0, u \cdot \nu = 0 \ \mbox{on} \ \partial \Omega \}
$$
and consider the following controlled Keller-Segel-Navier-Stokes  coupled system
\begin{equation}\label{system}
\left |
\begin{array}{ll}
n_{t}  + u\cdot \nabla n  - \Delta n  = -\nabla \cdot (n\nabla c)   &     \mbox{in}  \  \ Q,  \\
c_t  + u\cdot \nabla c - \Delta c  = -nc  + g_1 \chi_{1}  &     \mbox{in}  \  \ Q, \\
u_t - \Delta u + (u\cdot \nabla) u + \nabla p = ne_N + g_2e_{N-2} \chi_{2} &     \mbox{in}  \  \ Q,  \\
\nabla \cdot u =0 &     \mbox{in}  \  \ Q, \\
\frac{\partial n}{\partial \nu} = \frac{\partial c}{\partial \nu} = 0; \ u = 0    &    \mbox{on}  \  \   \Sigma, \\
n(x,0) = n_0;  \ c(x,0) = c_0; \ u(x,0) = u_0  &    \mbox{in}    \  \  \Omega,
\end{array}
\right.
\end{equation}
where $g_1$ and $g_2$  are internal controls and the $\chi_i : \mathbb{R}^N \rightarrow \mathbb{R}$, $i=1,2$,  are $C^\infty$  functions  such that $supp \  \chi_i \subset \subset \omega_i$, $0 \leq \chi_i \leq 1$ and $\chi _i\equiv 1$ in $\omega_i^{0}$, for some $ \emptyset \neq \omega_i^{0} \subset \subset \omega_i$, with $\omega_1^{0}\cap \omega_2^{0} \neq \emptyset$ when $N=3$, and

\begin{align}
e_0 =  (0,0), \ e_1=(1,0,0) \ \   \text{and} \ \ e_N&=  \left\{
\begin{array}{cc}
(0,1) & \text{if} \ N=2;\\
(0,0,1) &\text{if} \ N=3.
\end{array}
 \right.
 \end{align}
 The unknowns $n$, $c$, $u$ and $p$ are the cell density, substrate concentration, velocity and
pressure of the fluid, respectively.

System \eqref{system} was proposed by \textit{Tuval et al.} in \cite{Tuval} to describe large-scale convection patterns in a water drop sitting on a glass surface containing oxygen-sensitive bacteria, oxygen diffusing into the drop through the fluid-air interface (for more details see, for instance, \cite{Lorz2, LiuLorz,Lorz1}). In particular,  it is a good model for   the collective behavior of a suspension of oxygen-driven bacteria in an aquatic fluid, in which the oxygen concentration $c$ and the density of the bacteria $n$ diffuse and are transported by the fluid at the same time.

The main objective of this paper is to analyze the controllability problem of system \eqref{system} around some particular trajectories. More precisely,  we consider  $(M, M_0) \in \mathbb{R}_{+}^2$ and aim to find $g_1$ and $g_2$ such that the solution $(n,c,u, p)$ of \eqref{system} satisfies
 \begin{equation}\label{F1}
 n(T) = M; \ c(T) = M_0e^{-MT}; \ u(T) = 0.
 \end{equation}
 Moreover,  for the case $N=2$, we want to show that we can take $g_2\equiv0$.

\begin{remark}
 Noticing that $(n, c, u, p) = (M,M_0e^{-Mt},0, Mx_N)$ is a  solution of \eqref{system}, we see that \eqref{F1} means we are driving the solution \eqref{system} to a prescribed trajectory.
 \end{remark}

To analyze the controllability of system \eqref{system} around $(M,c_0e^{-Mt},0,  Mx_N)$, we first consider its linearization around this trajectory, namely
\begin{equation}\label{system-1}
\left |
\begin{array}{ll}
n_{t}    - \Delta n  = - M \Delta c + h_1   &     \mbox{in}  \  \ Q,  \\
c_t   - \Delta c  = -Mc  -M_0e^{-Mt}n + g_1 \chi_{\omega_1} +h_2  &     \mbox{in}  \  \ Q, \\
u_t - \Delta u  + \nabla p = ne_N + g_2\chi_{\omega_2}e_{N-2}  +H_3 &     \mbox{in}  \  \ Q,  \\
\nabla \cdot u =0 &     \mbox{in}  \  \ Q, \\
\frac{\partial n}{\partial \nu} = \frac{\partial c}{\partial \nu} = 0; \ u = 0    &    \mbox{on}  \  \   \Sigma, \\
n(x,0) = n_0;  \ c(x,0) = c_0; \ u(x,0) = u_0  &    \mbox{in}    \  \  \Omega,
\end{array}
\right.
\end{equation}
where the functions $h_1$ and  $h_2$ and the vector function $H_3$  are given exterior forces such that $(h_1, h_2, H_3)$ belongs  to an appropriate Banach space X (see \eqref{X}).  Our objective will be to find $g_1$ and $g_2$ such that the solution $(n,c,u,p)$ satisfies $n(T)=0$, $c(T)=0$ and $u(T)=0$. Moreover we want that   $\bigl(u\cdot \nabla n + \nabla \cdot (n\nabla c), nc + u\cdot \nabla c, (u\cdot \nabla ) u \bigl)$ belongs to $X$. Then we employ an inverse mapping argument introduced in \cite{Im-1} to obtain the controllability of \eqref{system} around $(M,c_0e^{-Mt},0,  Mx_N)$.

It is well-known that the null controllability of \eqref{system-1} is equivalent to a suitable observability inequality for  the solutions of  its  adjoint system
\begin{equation}\label{adjoint-1}
\left |
\begin{array}{ll}
-\varphi_{t}    - \Delta \varphi =   -M_0e^{-Mt}\xi + ve_N+ f_1   &     \mbox{in}  \  \ Q,  \\
-\xi_t   - \Delta \xi  =  -M\xi - M \Delta \varphi  +f_2&     \mbox{in}  \  \ Q, \\
-v_t - \Delta v  + \nabla \pi = F_3 &     \mbox{in}  \  \ Q,  \\
\nabla \cdot v =0 &     \mbox{in}  \  \ Q, \\
\frac{\partial \varphi}{\partial \nu} = \frac{\partial \xi}{\partial \nu} = 0; \ v = 0    &    \mbox{on}  \  \   \Sigma, \\
\varphi(x,T) = \varphi_T;  \ \xi(x,T) = \xi_T; \ v(x,T) = v_T  &    \mbox{in}    \  \  \Omega, \\
\int_{\Omega} \varphi_T(x)dx=0,
\end{array}
\right.
\end{equation}
where $(f_1, f_2, F_3) \in L^2(Q)\times L^2(Q) \times L^2(0,T;\bold{V})$. In this work, we obtain the  observability inequality as a consequence of an  appropriate  global Carleman inequality for the solution of \eqref{adjoint-1}.

With the help of the Carleman inequality that we obtain for the solutions of \eqref{adjoint-1} and an appropriate inverse function theorem, we will prove the following result, which is the main result of this paper.

\begin{theorem}\label{mainresultt}
Let $(M,M_0) \in \mathbb{R}^2_+$ and   $(n_0,c_0, u_0) \in H^1(\Omega) \times H^2(\Omega) \times \bold{V}$,  with  $n_0, c_0 \geq 0$,  $\frac{1}{|\Omega|}\int_\Omega n_0dx =M$ and  $\frac{\partial c_0}{\partial \nu} = 0$ on $\partial \Omega$.  We have

\begin{itemize}
\item If $N=2$, there exists $\gamma >0$ such that  if $||(n_0-M, c_0-M_0e^{-MT}, u_0)||_{H^1(\Omega) \times H^2(\Omega)\times \bold{V}} \leq \gamma$, we can find $g_1 \in L^2(0,T;H^1(\Omega))$, and an associated solution $(n,c, u, p)$ to \eqref{system}  satisfying
$$
(n(T),c(T), u(T)) = (M,M_0e^{-MT}, 0) \ \text{in} \ \Omega.
$$

\item If $N=3$, there exists $\gamma >0$ such that  if $||(n_0-M, c_0-M_0e^{-MT}, u_0)||_{H^1(\Omega) \times H^2(\Omega)\times \bold{V}} \leq \gamma$,  we can find $g_1 \in L^2(0,T;H^1(\Omega))$ and $g_2  \in L^2(0,T;L^2(\Omega))$ and an associated solution $(n,c, u, p)$ to \eqref{system}  satisfying
$$
(n(T),c(T), u(T)) = (M,M_0e^{-MT}, 0) \ \text{in} \ \Omega.
$$
\end{itemize}
\end{theorem}

\null

\begin{remark}
Assumption  $\frac{1}{|\Omega|}\int_\Omega n_0dx =M$ in Theorem \ref{mainresultt} is a necessary condition for the controllability of system \eqref{system}. This is due to the fact that the  mass of $n$ is preserved, i.e.,
\begin{equation}\label{massa}
\frac{1}{|\Omega|} \int_{\Omega}n(x,t)dx = \frac{1}{|\Omega|} \int_{\Omega}n_0(x)dx,    \ \  \forall  t>0.
 \end{equation}
\end{remark}

In the two dimensional case, because we want to take $g_2=0$, we only have a control acting  on the second equation of \eqref{system-1}. Therefore, in the Carleman inequality for the solutions of   \eqref{adjoint-1},   we need to bound global integrals of $\varphi$ and $\xi$ and $v$  in terms of a local integral of $\xi$ and global integrals of $f_1$, $f_2$ and $F_3$.

 For the three dimensional case, we have two controls, $g_1$ acting on  $\eqref{system}_2$  and another control $g_2$ acting on the third component of the Navier-Stokes equation  $\eqref{system}_3$. In this case,  in the Carleman inequality for the solutions of   \eqref{adjoint-1},   we need to bound global integrals of $\varphi$ and $\xi$ and $v$  in terms of a local integral of $\xi$ another in $v_3$  and global integrals of $f_1$, $f_2$ and $F_3$.

 For both cases, $N=2$ or $3$, the main difficulty when proving the desired Carleman inequality for solutions of \eqref{adjoint-1} comes from the fact that the coupling in the second equation is in $\Delta \varphi$ and not in $\varphi$.

Concerning the controllability of system \eqref{system}, we are not aware of any controllability result obtained previously to Theorem \ref{mainresultt}. For the controllabity of the Keller-Segel system with control acting on the component of the chemical, as far as we know, the only result is the one in  \cite{CSG}, where  the local controllability of the Keller-Segel system around a constant trajectory is obtained. On the other hand,  for the Navier-Stokes equations, controllability  has been the object of intensive research during the past few years  and several  local controllability results has been obtained in many different contexts (see, for instance, \cite{CorL, FC-G-P, Im-2} and references therein).

It is important to say that it is not possible to combine the result in  \cite{CSG} with any previous controllability result for the Navier-Stokes system in order to obtain controllability results for  \eqref{system}. In fact, for the first two equations in \eqref{adjoint-1},  one cannot use the Carleman inequality obtained in \cite{CSG}.  This is due to the fact that for the obtainment of a suitable Carleman inequality for the adjoint system in \cite{CSG}, it is necessary  that $\frac{\partial \Delta \varphi}{\partial \nu}=0$, which is no longer the case for \eqref{adjoint-1}. For this reason,  to deal with  the chemotaxis part of system \eqref{adjoint-1}, we borrow some ideas from \cite{FWCSthesis}. For the Stokes part of \eqref{adjoint-1}, it is also not possible to use Carleman inequalities for the Stokes system obtained in previous works as in  \cite{Car-Guey} and \cite{CorGue}. Indeed, since in \eqref{adjoint-1} the coupling in the second equation is in $\Delta \varphi$, and we have a term in $ve_N$ in the first equation, for the Stokes equation, we need to show a Carleman inequality with a local term in $\Delta v e_N$. Actually, in \cite{Car-Guey} a Carleman inequality  for the Stokes system  with measurement through  a local observation in the Laplacian of one component  is proved. However, that result cannot be  used in our situation (see  Remark \ref{remarkMamdou}). For this reason,  we need to prove a new local Carleman inequality for solutions of the Stokes system (see  Lemma \ref{CarlemanStokes}).

 \vskip0.3cm

This paper is divided as follows. Section \ref{sect2} is devoted to prove a suitable observability inequality for the  solutions of  \eqref{adjoint-1}. In Section \ref{sect3}, we  prove the null controllability of  system \eqref{system-1}, with an appropriate right-hand side. Finally, in Section \ref{sec4} we prove Theorem \ref{mainresultt}.


\section{Carleman inequality}\label{sect2}

In this section we prove a Carleman inequality for the adjoint system \eqref{adjoint-1}. This inequality will be the main ingredient for the obtention of a  controllability result for the nonlinear system \eqref{system} in the next section.

We begin introducing  several weight functions which we need to state our Carleman inequality.  The basic weight  will be a function $\eta_0 \in C^2(\overline{\Omega})$ verifying
$$ \eta_0(x) > 0 \ \mbox{in} \  \Omega , \  \  \eta_0 \equiv 0 \  \mbox{on} \ \partial \Omega, \   \  |\nabla \eta_0(x)| > 0 \ \forall x \in \overline{\Omega \backslash \omega_0},$$
where $\omega_0$  is a nonempty open set with

\begin{align}
\omega_0 &\subset \subset  \left\{
\begin{array}{cc}
\omega_1^0  & \text{if} \ N=2;\\
\omega_1^0 \cap \omega_2^0 &\text{if} \ N=3.
\end{array}
 \right.
 \end{align}

The existence of such a function $ \eta_0$ is proved in \cite{F-Im}.

For some positive real number  $\lambda$, we introduce:
\begin{align} \label{weightfunctions1}
&\phi(x,t) = \frac{e^{\lambda \eta_0(x)}}{\ell(t)^{11}}, \ \alpha(x,t) = \frac{e^{\lambda \eta_0(x)} - e^{2\lambda||\eta_0||_{\infty}}}{\ell(t)^{11}}, \nonumber \\
&\widehat{\phi}(t) = \min_{x \in \overline{\Omega}} \phi(x,t), \ \phi^*(t) = \max_{x \in \overline{\Omega}} \phi(x,t), \ \alpha^*(t) = \max_{x \in \overline{\Omega}}\alpha (x,t),\ \widehat{\alpha} = \min_{x \in \overline{\Omega}}\alpha (x,t),
\end{align}
where $\ell \in C^{\infty}([0,T])$ is a positive function satisfying
$$
\ell(t) =t \ \text{for} \  t\in [0,T/4], \ \ell(t)= T-t  \ \text{for} \  t\in [3T/4,T], \   \text{and}
$$
$$
\ell(t)  \leq \ell(T/2), \forall t \in [0,T].
$$

\begin{remark}
From the definition of $\phi$ and $\widehat{\phi}$, it follows that
\begin{equation*}
\widehat{\phi}(t) \leq \phi (x,t) \leq e^{\lambda  \|\eta_0\|_{\infty}} \widehat{\phi}(t),
\end{equation*}
for every $x\in \Omega $, every $t\in [0,T]$ and every   $\lambda \in \mathbb{R}_{+}$.
\end{remark}

We also introduce the following notation:
\begin{align}\label{ineq-epsxi-1}
\widehat{I}_{\beta}(s; q):= & \ s^{3+\beta} \iint\limits_Q e^{2s\alpha}\phi^{3+\beta}|q|^2 dxdt + s^{1+\beta} \iint\limits_Q e^{2s\alpha}\phi^{1+\beta} |\nabla q|^2 dxdt,
\end{align}
\begin{align}\label{ineq-epsxi}
I_{\beta}(s; q):= & \ \widehat{I}_{\beta}(s; q)  + s^{-1+\beta} \iint\limits_Q e^{2s\alpha}\phi^{-1+\beta}(|q_t|^2 + |\Delta q |^2) dxdt,
\end{align}
where $\beta$ and $s$ are   real numbers and $q = q(x,t)$.

\null

The main result of this section is the following Carleman estimate for the solutions of \eqref{adjoint-1}.
\begin{theorem}\label{CarlemanP}
There exist $C = C(\Omega, \omega_0)$ and  $\lambda_0 = \lambda_0(\Omega, \omega_0)$ such that,  for every $\lambda \geq \lambda_0$, there exists $s_0 = s_0(\Omega, \omega_0, \lambda,T)$ such that, for any $s \geq s_0$, any $(\varphi_T, \xi_T,v_T) \in L^2(\Omega)\times L^2(\Omega) \times \bold{H}$ and any $(f_1, f_2,F_3) \in L^2(Q)\times L^2(Q)\times L^2(0,T;\bold{V})$, the solution $(\varphi, \xi,v)$ of system $(\ref{adjoint-1})$ satisfies

\begin{align}\label{Carlemanadjointsystem0}
s^5\int \! \! \! \int_Q&e^{2s\widehat{\alpha}}\widehat{\phi}^5 | z_2|^2 dx dt + s^5 \int \! \! \! \int_Q e^{5s\widehat{\alpha}}\widehat{\phi}^5|v|^2 dx dt\nonumber \\
+& \sum_{i \neq 2}\biggl(s^5\int \! \! \! \int_Q  e^{2s\alpha}\phi^5|\Delta z_i|^2dxdt + s^3\int \! \! \! \int_Q e^{2s\alpha}\phi^3| \nabla \Delta z_i|^2dx dt+  \widehat{I}_{-2}(s; \nabla \nabla \Delta z_i) \biggr)  \\
&  + \widehat{I}_0(s, \Delta \psi)+I_2(s, e^{\frac{3}{2}s\widehat{\alpha}}\widehat{\phi}^{-9/2}\xi) +\int \! \! \! \int_Q e^{2s\alpha+3s\widehat{\alpha}}\widehat{\phi}^{-6}|\Delta \varphi|^2dxdt \nonumber \\
&+\int \! \! \! \int_Q e^{5s\widehat{\alpha}}\widehat{\phi}^{-6}|\nabla \varphi|^2dxdt  \nonumber \\
&\leq  C\biggl(s^{33}\int \! \! \! \int_{\omega_1 \times (0,T)} e^{2s\alpha+3s\widehat{\alpha}} \widehat{\phi}^{61}|\chi_1|^2|\xi|^2 dxdt  +(N-2)s^9\int \! \! \! \int_{\omega_2 \times (0,T) } e^{2s\alpha}\phi^9|\chi_2|^2| v_1|^2 dxdt\nonumber \\
&+  \int \! \! \! \int_Q  e^{3s\widehat{\alpha}} \widehat{\phi}^{-9}|f_1|^2 dxdt  + s^{15}\int \! \! \! \int_{Q} e^{2s\alpha+3s\widehat{\alpha}} \widehat{\phi}^{24}   |f_2|^2dxdt  +\|e^{\frac{3}{2}s\widehat{\alpha}} F_3\|^2_{L^2(0,T;\bold{V})}\biggl).
\end{align}

\end{theorem}

\null

We prove Theorem \ref{CarlemanP} in the case $N=3$ and, with the due adaptations, the case $N=2$ is performed in the exact same way.

\null

The plan of the proof contains five parts:

\null

\begin{itemize}
\item[\textit{Part 1}.] \textit{Carleman inequality for $v$:} We write $e^{\frac{3}{2}s\widehat{\alpha}}v = w+z$, where $w$ solves, together with some $q$,  a Stokes system with right-hand side in $L^2(0,T;\bold{V})$ and $z$ solves, together with some $r$, a Stokes system with right-hand side in $L^2(0,T; \bold{H}^3(\Omega))\cap H^1(0,T; \bold{V})$. Applying regularity estimates for $w$ and a Carleman estimate for $z$, we obtain a Carleman inequality for $v$ in terms of  local integrals of $\Delta z_1$ and  $\Delta z_3$ and a global integral in $F_3$.

\null

\item[\textit{Part 2}.] \textit{Carleman inequality for $\Delta \varphi$:} We write $e^{\frac{3}{2}s\widehat{\alpha}}\widehat{\phi}^{-9/2}\varphi =\eta + \psi$, where  $\eta$ solves a heat equation with a $L^2$ right-hand side and $\psi$  solves a heat equation with  right-hand side in $L^2(0,T;H^2(\Omega))\cap H^1(0,T;L^2(\Omega))$. Applying a Carleman inequality for $\psi$ and regularity estimates for $\eta$ we obtain a global estimate of $\Delta \varphi$ in terms of a local integral of $\Delta \psi$ and  global integrals of  $\Delta \xi$, $\Delta v_3$ and $f_1$.

\null

\item[\textit{Part 3}.] \textit{Carleman inequality for $\xi$:} Using $\eqref{adjoint-1}_2$, we obtain a Carleman estimate for the function $e^{\frac{3}{2}s\widehat{\alpha}}\widehat{\phi}^{-9/2}\xi$. Combining this inequality  with the Carleman inequality from the previous step,  global estimates of $\xi$ and $\Delta \varphi$ in terms of  local integrals of $\xi$ another in $\Delta \psi$ and global integrals of $\Delta v_3$, $f_1$ and $f_2$ are obtained.

\null

\item[\textit{Part 4}.] \textit{Estimate of $\Delta z_3$:} Using $\eqref{adjoint-1}_1$,  we estimate a local integral in $\Delta z_3$ in terms of  local integrals of $\xi$ and $\Delta \psi$  and some lower order terms.

\null

\item[\textit{Part 5}.] \textit{Estimate of $\Delta \psi$:} In the last part, we use $\eqref{adjoint-1}_2$ to estimate a local integral of $\Delta \psi$ in terms of a local integral of $\xi$ and global integrals in $f_1$ and $f_2$.

\end{itemize}

\null

Along the proof, for  $k\in \mathbb{R}$ and a vector function $F$ with $m$-coordinates, we  write
$$
\| F\|_{L^2(0,T;\bold{H}^k(\Omega))} := \| F\|_{L^2(0,T; H^k(\Omega)^m)}
$$
and
$$
\| F\|_{L^2(0,T;\bold{H}^k(\partial \Omega))} := \| F\|_{L^2(0,T; H^k(\partial \Omega)^m)}.
$$
and, for every $p\geq 0$
$$
\| F\|_{\bold{W}^{k,p}(\Sigma)}= \| F\|_{W^{k,p}(\Sigma)^m}.
$$

We will also denote $\omega^j_0$, $j \in \mathbb{N}^*$, to represent subsets
$$
\omega_0:= \omega^0_0 \subset \subset \omega^1_0 \subset \subset \omega^2_0 \subset \subset \cdots \subset \subset \omega_1\cap \omega_2
$$
and, for a fixed  $j \in \mathbb{N}^*$,  we will denote by $\theta_j$ a  function in $C^{\infty}_0(\omega^j_0)$ such that
\begin{align}\label{cutoffdef}
0 \leq\theta_j\leq 1 \  \mbox{and}  \ \theta_j \equiv 1\ \mbox{on} \  \omega^{j-1}_0.
\end{align}

\null

\begin{proof}[Proof of Theorem \ref{CarlemanP}]
For an easier comprehension, the proof is divided  into several steps.

\null

\textbf{Step 1}: \textit{Carleman estimate for $v$}.

\null

 Let us consider  $\rho(t):=e^{\frac{3}{2}s\widehat{\alpha}}$
and write
 \begin{equation}\label{w+z}
 (\rho v, \rho \pi) = (w,q) + (z,r),
 \end{equation}
 where  $(w,q)$ and $(z,r)$ are the solutions of
\begin{equation}\label{eq:u1}
\left\lbrace \begin{array}{ll}
    -w_t - \Delta w + \nabla q = \rho F_3 & \mbox{in} \ Q, \\
    \nabla\cdot w = 0 & \mbox{in} \ Q, \\
    w = 0 & \mbox{on} \ \Sigma, \\
    w(T) = 0 & \mbox{in} \ \Omega,
 \end{array}\right.
\end{equation}
and
\begin{equation}\label{eq:u2}
\left\lbrace \begin{array}{ll}
    -z_t - \Delta z + \nabla r = -\rho' v & \mbox{in} \ Q, \\
    \nabla\cdot z = 0 & \mbox{in} \ Q, \\
    z = 0 & \mbox{on} \ \Sigma, \\
    z(T) = 0 & \mbox{in} \ \Omega,
 \end{array}\right.
\end{equation}
respectively.

For $w$,  Lemma \ref{regStokes10} yields
\begin{equation}\label{eq:regularity}
\|w\|^2_{L^2(0,T;\bold{H}^3(\Omega))} + \|w\|^2_{H^1(0,T;\bold{V})}\leq C \|\rho F_3\|^2_{L^2(0,T;\bold{V})}.
\end{equation}


For $z$, we prove the following Carleman estimate.

\begin{lemma}\label{CarlemanStokes}
There exist $C = C(\Omega, \omega_0)$ and  $\lambda_0 = \lambda_0(\Omega, \omega_0)$ such that,  for every $\lambda \geq \lambda_0$, there exists $s_0 = s_0(\Omega, \omega_0, \lambda,T)$ such that
\begin{align}\label{Z11}
s^5\int \! \! \! \int_Q&e^{2s\widehat{\alpha}}\widehat{\phi}^5 | z_2|^2 dx dt + s^5 \int \! \! \! \int_Q e^{2s\widehat{\alpha}}\widehat{\phi}^5|\rho|^2|v|^2 dx dt\nonumber \\
+& \sum_{i=1,3}\biggl(s^5\int \! \! \! \int_Q  e^{2s\alpha}\phi^5|\Delta z_i|^2dxdt + s^3\int \! \! \! \int_Q e^{2s\alpha}\phi^3| \nabla \Delta z_i|^2dx dt+  \widehat{I}_{-2}(s; \nabla \nabla \Delta z_i) \biggr)  \\
 & \leq \ C\sum_{i=1,3}\biggl( \|\rho F_3\|^2_{L^2(0,T;\bold{V})}  + s^5\int \! \! \! \int_{\omega_0^3 \times (0,T) } e^{2s\alpha}\phi^5|\Delta z_i|^2 dxdt
 \biggr).\nonumber
\end{align}
\end{lemma}

\begin{remark} \label{remarkMamdou}
A similar result to Lemma \ref{CarlemanStokes} was obtained in \cite[Proposition 3.2]{Car-Guey}. However, we cannot apply that result to system \eqref{eq:u2} because it would give a global term in $z_2$ in the right hand-side which could not be absorbed by the left hand-side of the inequality. Moreover,  the regularity required for the vector function $F_3$ is not as optimal as in Lemma \ref{CarlemanStokes}.  For this reason, we give the proof of Lemma \ref{CarlemanStokes} in the Appendix \ref{proofcarlemanstokes}.
\end{remark}
\null


\textbf{Step 2.} \textit{Carleman inequality for $\Delta \varphi$.}

\null

We write $\rho\widehat{\phi}^{-9/2}\varphi =\eta + \psi$, where the functions  $\eta$ and $\psi$ stand to solve
\begin{equation}\label{x3}
\left |
\begin{array}{ll}
-\eta_t  - \Delta \eta  = \rho\widehat{\phi}^{-9/2}f_1   & \mbox{in}  \  \ Q,  \\
\frac{\partial \eta}{\partial \nu} =0 &    \mbox{on}  \  \   \Sigma, \\
\eta(T) = 0 &    \mbox{in}    \  \  \Omega
\end{array}
\right.
\end{equation}
and
\begin{equation}\label{x2-KS}
\left |
\begin{array}{ll}
-\psi_t  - \Delta \psi  = -M_0e^{-Mt}\rho\widehat{\phi}^{-9/2}\xi  +\rho\widehat{\phi}^{-9/2}v_3 - (\rho\widehat{\phi}^{-9/2})_t\varphi &     \mbox{in}  \  \ Q,  \\
\frac{\partial \psi}{\partial \nu} =0 &    \mbox{on}  \  \   \Sigma, \\
\psi(T) = 0 &    \mbox{in}    \  \  \Omega,
\end{array}
\right.
\end{equation}
respectively.

Using standard regularity estimates for the heat equation with Neumann boundary conditions, we have
\begin{equation}\label{energy-neu}
\|\eta \|^2_{H^1(0,T; L^2(\Omega))} +\|\eta\|^2_{L^2(0,T;H^2(\Omega))} \leq C \|\rho\widehat{\phi}^{-9/2}f_1\|^2_{L^2(Q)},
\end{equation}
for some $C>0$.

Next, from (\ref{x2-KS}) we see that
\begin{equation}\label{x2-1}
\left |
\begin{array}{ll}
-(\Delta\psi)_t  - \Delta (\Delta \psi)  = -M_0e^{-Mt}\rho\widehat{\phi}^{-9/2}\Delta \xi  +\rho\widehat{\phi}^{-9/2}\Delta v_3 - (\rho\widehat{\phi}^{-9/2})_t \Delta \varphi &     \mbox{in}  \  \ Q,  \\
\frac{\partial (\Delta \psi)}{\partial \nu} =\rho\widehat{\phi}^{-9/2}\frac{\partial v_3}{\partial \nu}  &    \mbox{on}  \  \   \Sigma, \\
\Delta \psi(T) = 0 &    \mbox{in}    \  \  \Omega.
\end{array}
\right.
\end{equation}
Applying \cite[Theorem 1]{C-G-B-P-1}, we have the following estimate
\begin{align}\label{x4-KS}
\widehat{I}_{0}(s, \Delta \psi ) \leq & \ C \biggl( s^3 \int \! \! \! \int_{\omega^4_0 \times (0,T)} e^{2s\alpha} \phi^3|\Delta \psi |^2dxdt + \int \! \! \! \int_Q e^{2s\alpha}\widehat{\phi}^{-9}|\rho|^2(|\Delta \xi|^2+ |\Delta v_3|^2)dxdt \nonumber \\
  & +s\int \! \! \! \int_{\Sigma} e^{2s\alpha} \widehat{\phi}^{-8}|\rho|^2|\frac{\partial v_3}{\partial \nu}|^2d\sigma dt  + s^{2+2/11} \int \! \! \! \int_Q e^{2s\alpha} \widehat{\phi}^{-6}|\rho|^2|\Delta \varphi|^2dxdt  \biggl),
\end{align}
for  any $s \geq s_0(\Omega, \omega_0, T,  \lambda)$ (a proof of  \eqref{x4-KS} is achieved taking into account that
$$|(\rho\widehat{\phi}^{-9/2})_t|  \leq Cs^{1+1/11} \widehat{\phi}^{-3}\rho, $$
since $$|\widehat{\alpha}_t| + |\widehat{\phi}_t| \leq C T \widehat{\phi}^{12/11}  \ \text{and}  \  |\widehat{\phi}^{-1}| \leq CT^{22}, $$
 for some $C = C(\Omega, \omega_0, \lambda)$ and any $ s \geq s_0(\Omega, \omega_0, \lambda, T)$).

 \null

Because $\rho\widehat{\phi}^{-9/2}\Delta \varphi = \Delta \psi +\Delta \eta$, estimate \eqref{x4-KS} gives
\begin{align}\label{x5-KS}
\widehat{I}_0(s, \Delta \psi ) \leq & C \biggl(  s^3\int \! \! \! \int_{\omega^4_0 \times (0,T)} e^{2s\alpha} \phi^3|\Delta \psi |^2dxdt + \int \! \! \! \int_Qe^{2s\alpha}\widehat{\phi}^{-9}|\rho|^2(|\Delta \xi|^2+ |\Delta v_3|^2)dxdt \nonumber \\
  &+s\int \! \! \! \int_{\Sigma} e^{2s\alpha} \widehat{\phi}^{-8}|\rho|^2|\frac{\partial v_3}{\partial \nu}|^2d\sigma dt   + s^{2+2/11}  \int \! \! \! \int_Q e^{2s\alpha} \phi^{3}(|\Delta \psi|^2 + |\Delta \eta|^2) dxdt  \biggl).
\end{align}
The last term on the right-hand side of \eqref{x5-KS}  can be estimated as follows
\begin{align}\label{x5-1}
s^{2+2/11} \int \! \! \! \int_Q e^{2s\alpha} \phi^{3}(|\Delta \psi|^2 + |\Delta \eta|^2) dxdt &\leq  C \|\rho\widehat{\phi}^{-9/2}f_1\|^2_{L^2(Q)}+\delta \widehat{I}_0(s, \Delta \psi ) ,
\end{align}
for any $\delta >0$ and  any $s \geq s_0(\Omega, \omega_0, T,  \lambda)$. Here we have used estimate \eqref{energy-neu} and the definition of $\widehat{I}_0(s, \Delta \psi )$.

Therefore, combining (\ref{energy-neu}),  (\ref{x5-KS}), (\ref{x5-1}),  we obtain
\begin{align}\label{x5-b}
 \widehat{I}_0(s, \Delta \psi )   + & \|\eta\|^2_{H^1(0,T; L^2(\Omega))} +\|\eta\|^2_{L^2(0,T;H^2(\Omega))} +s^3\int \! \! \! \int_Q e^{2s\alpha}\widehat{\phi}^{-6}|\rho|^2|\Delta \varphi|^2dxdt \nonumber \\
\leq  & \ C \biggl( s^3 \int \! \! \! \int_{\omega_0^4 \times (0,T)} e^{2s\alpha} \phi^3|\Delta \psi |^2dxdt + \int \! \! \! \int_Q e^{2s\alpha}\widehat{\phi}^{-9}|\rho|^2(|\Delta \xi|^2+ |\Delta v_3|^2)dxdt \nonumber \\
  &+s\int \! \! \! \int_{\Sigma} e^{2s\alpha} \widehat{\phi}^{-8}|\rho|^2|\frac{\partial v_3}{\partial \nu}|^2d\sigma dt  + \|\rho\widehat{\phi}^{-9/2}f_1\|^2_{L^2(Q)}\biggl).
\end{align}

\null

\textbf{Step 3.}  \textit{Carleman inequality for $\xi$.}

\null


We consider the function $\rho\widehat{\phi}^{-9/2}\xi$, which fulfills the following system:
\begin{equation}\label{y1}
\left |
\begin{array}{ll}
-(\rho\widehat{\phi}^{-9/2}\xi)_{t}  - \rho\widehat{\phi}^{-9/2}\Delta \xi +M\rho\widehat{\phi}^{-9/2}\xi  = \tilde{f}_2  &     \mbox{in}  \  \ Q,  \\
\frac{\partial  (\rho\hat{\phi}^{-9/2}\xi)}{\partial \nu} =0 &    \mbox{on}  \  \   \Sigma, \\
(\rho\widehat{\phi}^{-9/2}\xi)(T) = 0 &    \mbox{in}    \  \  \Omega,
\end{array}
\right.
\end{equation}
with  $\tilde{f}_2=  - M\rho\widehat{\phi}^{-9/2}\Delta \varphi - (\rho\widehat{\phi}^{-9/2})_t\xi + \rho\widehat{\phi}^{-9/2}f_2$.
 \\
From  Lemma \ref{lemma-4-CSG}, we have the estimate
\begin{align}\label{y2}
I_{2}(s, \rho\widehat{\phi}^{-9/2}\xi )  &\leq C\biggl(  s^5\int \! \! \! \int_{\omega_0^5 \times (0,T)}e^{2s\alpha}\widehat{\phi}^{-4}|\rho|^2|\xi|^2dxdt  + s^{4+2/11}\int \! \! \! \int_Q \widehat{\phi}^{-4}e^{2s\alpha}|\rho|^2 |\xi|^2  dxdt   \\
&   + s^2\int \! \! \! \int_Q \phi^2 e^{2s\alpha} (|\Delta \psi|^2 + |\Delta \eta|^2) dxdt+ s^2\int \! \! \! \int_Q  e^{2s\alpha}\widehat{\phi}^{-7} |\rho|^2|f_2|^2 dxdt \biggl).   \nonumber
\end{align}
Here we have used the fact that $|(\rho\widehat{\phi}^{-9/2})_t|  \leq Cs^{1+1/11} \widehat{\phi}^{-3}\rho$.

Using  estimate \eqref{x5-1} and the definition of $\widehat{I}_0(s, \Delta \psi )$, we see that
\begin{align}\label{y2}
I_{2}(s,\rho\widehat{\phi}^{-9/2}\xi )  & \leq C\biggl(  s^5\int \! \! \! \int_{\omega_0^5 \times (0,T)}e^{2s\alpha}\widehat{\phi}^{-4}|\rho|^2|\xi|^2dxdt     \\
 & + \int \! \! \! \int_Q \widehat{\phi}^{-9} |\rho|^2|f_1|^2dxdt + s^2\int \! \! \! \int_Q e^{2s\alpha } \widehat{\phi}^{-7}|\rho|^2|f_2|^2 dxdt \biggl) +\delta \widehat{I}_0(s,\Delta \psi) \nonumber,
\end{align}
for any $\delta >0$ and any  $s \geq s_0(\Omega, \omega_0,T, \lambda)$.


Adding \eqref{x5-b} and \eqref{y2},  absorbing the lower order terms,  we obtain
\begin{align}\label{xm}
&I_{2}(s, \rho\widehat{\phi}^{-9/2}\xi )+ \widehat{I}_0(s, \Delta \psi )+s^3\int \! \! \! \int_Q e^{2s\alpha}\widehat{\phi}^{-6}|\rho|^2|\Delta \varphi|^2dxdt  \nonumber \\
& \leq  C \biggl(  s^3\int \! \! \! \int_{\omega_0^4 \times (0,T)}e^{2s\alpha} \phi^3|\Delta \psi |^2dxdt + s^5\int \! \! \! \int_{\omega_0^5 \times (0,T)}e^{2s\alpha}\phi^{-4}|\rho|^2|\xi|^2dxdt   \nonumber \\
&  +\int \! \! \! \int_Q \widehat{\phi}^{-9} |\rho|^2|f_1|^2 dxdt + s^2\int \! \! \! \int_Q  e^{2s\alpha}\widehat{\phi}^{-7}|\rho|^2 |f_2|^2 dxdt \nonumber \\
 &+  \int \! \! \! \int_Q e^{2s\alpha}\widehat{\phi}^{-9} |\rho|^2|\Delta v_3|^2dxdt +s\int \! \! \! \int_{\Sigma} e^{2s\alpha} \widehat{\phi}^{-8}|\rho|^2|\frac{\partial v_3}{\partial \nu}|^2d\sigma dt\biggl),
\end{align}
for  any $s \geq s_0(\Omega, \omega_0,T, \lambda)$.

\null


\textbf{Step 4.} \textit{Estimate of a local integral of $\Delta z_3$.}

\null

In this step we estimate the local integral of $\Delta z_3$ in the right-hand side of \eqref{Z11} in Lemma \ref{CarlemanStokes}.

We begin using \eqref{x2-1} to see that
\begin{align}\label{estlocz_3}
&s^5\int \! \! \! \int_{\omega_0^3 \times (0,T) } e^{2s\alpha}\phi^5|\Delta z_3|^2 dxdt \nonumber \\
& = s^5\int \! \! \! \int_{\omega_0^3 \times (0,T) } e^{2s\alpha}\phi^5\Delta z_3\biggl( -\widehat{\phi}^{9/2}\bigl((\Delta\psi)_t  +\Delta (\Delta \psi)-(\rho\widehat{\phi}^{-9/2})_t \Delta \varphi\bigl) +M_0e^{-Mt}\rho\Delta \xi  -\Delta w_3\biggl) dxdt.
\end{align}
We estimate each one of the terms in the right-hand side of \eqref{estlocz_3}.

The first term is estimated as follows:
\begin{align}\label{AX0}
s^5\int \! \! \! \int_{\omega_0^3 \times (0,T) }& e^{2s\alpha}\phi^5\widehat{\phi}^{9/2}\Delta z_3(\Delta\psi)_t dxdt \nonumber \\
&=-s^5\int \! \! \! \int_{\omega_0^3 \times (0,T)} (e^{2s\alpha}\phi^5\widehat{\phi}^{9/2})_t \Delta z_3 \Delta\psi dxdt -s^5\int \! \! \! \int_{\omega_0^3 \times (0,T) } e^{2s\alpha}\phi^5\widehat{\phi}^{9/2}(\Delta z_3)_t \Delta\psi dxdt.
\end{align}
We have
\begin{align}\label{AX1}
|s^5\int \! \! \! \int_{\omega_0^3 \times (0,T)} &(e^{2s\alpha}\phi^5\widehat{\phi}^{9/2})_t \Delta z_3 \Delta\psi dxdt| \nonumber \\
&\leq  Cs^9\int \! \! \! \int_{\omega_0^4 \times (0,T)} e^{2s\alpha} \widehat{\phi}^{177/11}|\Delta \psi |^2dxdt  +\delta  s^5 \int \! \! \! \int_Q e^{2s\alpha}\widehat{\phi}^{5} |\Delta z_3|^2dxdt,
\end{align}
because
$$
|(e^{2s\alpha}\phi^5\widehat{\phi}^{9/2})_t| \leq Cs^{1+1/11}\widehat{\phi}^{6+1/11+9/2}e^{2s\alpha}.
$$
For the other term in \eqref{AX0}, we use \eqref{eq:u2} to see that
$$
-(\Delta z_3)_t - \Delta^2 z_3 = -\rho' \Delta v_3
$$
and write
\begin{align}\label{localestv31}
s^5\int \! \! \! \int_{\omega_0^3 \times (0,T) } e^{2s\alpha}\phi^5\widehat{\phi}^{9/2}(\Delta z_3)_t \Delta\psi dxdt =s^5\int \! \! \! \int_{\omega_0^3 \times (0,T) } e^{2s\alpha}\phi^5\widehat{\phi}^{9/2} \Delta\psi ( - \Delta^2 z_3 +\rho' \Delta v_3) dxdt.
\end{align}
Let us now estimate the terms on the right-hand side of \eqref{localestv31}.

It is not difficult to see that
\begin{align}\label{AX2}
|s^5\int \! \! \! \int_{\omega_0^3 \times (0,T) } e^{2s\alpha}\phi^5\widehat{\phi}^{9/2} \Delta\psi \Delta^2 z_3  dxdt| \leq Cs^9\int \! \! \! \int_{\omega_0^4 \times (0,T)} e^{2s\alpha} \widehat{\phi}^{18}|\Delta \psi |^2dxdt +    \delta \widehat{I}_{-2}(s, \nabla \nabla \Delta z_3).
\end{align}
We also have
\begin{align}\label{AX3}
|s^5\int \! \! \! \int_{\omega_0^3 \times (0,T) } &e^{2s\alpha}\phi^5\widehat{\phi}^{9/2} \Delta\psi \rho' \Delta v_3 dxdt | = |s^5\int \! \! \! \int_{\omega_0^3 \times (0,T) } e^{2s\alpha}\phi^5\widehat{\phi}^{9/2} \Delta\psi \rho' \rho^{-1} \Delta (z_3 + w_3)dxdt| \nonumber \\
&\leq C\biggl(s^9\int \! \! \! \int_{\omega_0^4 \times (0,T)} e^{2s\alpha} \widehat{\phi}^{178/11}|\Delta \psi |^2dxdt  + \|\rho F_3\|^2_{L^2(0,T;\bold{V})}\biggl) +\delta  s^5 \int \! \! \! \int_Q e^{2s\alpha}\widehat{\phi}^{5} |\Delta z_3|^2dxdt.
\end{align}
Here we have used estimate  \eqref{eq:regularity} and the fact that $ |\rho' \rho^{-1} | \leq Cs^{1+1/11} \widehat{\phi}^{1+1/11}$.

For the second term in \eqref{estlocz_3}, we have
\begin{align}\label{AX4}
 |s^5\int \! \! \! \int_{\omega_0^3 \times (0,T) }& e^{2s\alpha}\phi^5\widehat{\phi}^{9/2}\Delta z_3 \Delta (\Delta \psi) dxdt|  \\
\leq |s^5\int \! \! \! \int_{\omega_0^4 \times (0,T)}&\bigl( \Delta (\theta_4e^{2s\alpha}\phi^5\widehat{\phi}^{9/2}) \Delta z_3 +  2\nabla (\theta_4e^{2s\alpha}\phi^5\widehat{\phi}^{9/2}) \cdot \nabla\Delta z_3 + \theta_4e^{2s\alpha}\phi^5\widehat{\phi}^{9/2} \Delta^2 z_3 \bigl) \Delta \psi dxdt| \nonumber \\
 \leq   Cs^9&\int \! \! \! \int_{\omega_0^4 \times (0,T)}e^{2s\alpha} \widehat{\phi}^{18}|\Delta \psi |^2dxdt  \nonumber \\
&+\delta  \bigl(s^5 \int \! \! \! \int_Q e^{2s\alpha}\widehat{\phi}^{5} |\Delta z_3|^2dxdt+  s^3\int \! \! \! \int_Q e^{2s\alpha}\phi^3| \nabla \Delta z_3|^2dx dt + \widehat{I}_{-2}(s, \nabla \nabla \Delta z_3)\bigl), \nonumber
\end{align}
because
$$
|\nabla (\theta_4e^{2s\alpha}\phi^5\widehat{\phi}^{9/2})| \leq Cs\widehat{\phi}^{21/2}e^{2s\alpha}1_{\omega_0^4} \ \mbox{and}  \ |\Delta (\theta_4 e^{2s\alpha}\phi^5\widehat{\phi}^{9/2}) | \leq Cs^2\widehat{\phi}^{23/2}e^{2s\alpha}1_{\omega_0^4}.
$$

We estimate the other three terms in  \eqref{estlocz_3} as follows.

For the term in $\Delta \xi$, we use integration by parts to get
\begin{align} \label{AX6}
|s^5&\int \! \! \! \int_{\omega_0^3 \times (0,T) } e^{2s\alpha}\phi^5 \rho \Delta z_3 \Delta \xi   dxdt| \nonumber \\
&\leq s^5|\int \! \! \! \int_{\omega_0^4 \times (0,T) }\bigl(  \Delta( \theta_4e^{2s\alpha}\phi^5 \rho)\Delta z_3  + 2\nabla \Delta z_3  \cdot \nabla( \theta_4e^{2s\alpha}\phi^5 \rho) + \theta_4e^{2s\alpha}\phi^5 \rho \Delta^2 z_3   \bigl)\xi   dxdt| \nonumber \\
& \leq  C s^9\int \! \! \! \int_{\omega_0^4 \times (0,T) } e^{2s\alpha}\phi^9 |\rho|^2 |\xi|^2   dxdt+\delta\bigl(  s^5 \int \! \! \! \int_Q e^{2s\alpha}\phi^{5} |\Delta z_3|^2dxdt  +  s^3\int \! \! \! \int_Q e^{2s\alpha}\phi^3| \nabla \Delta z_3|^2dx dt\bigl)
\end{align}
because
$$
|\nabla( \theta_4e^{2s\alpha}\phi^5 \rho )| \leq Cs\phi^6 e^{2s\alpha} \rho 1_{\omega_0^4} \ \mbox{and}  \  | \Delta (\theta_4e^{2s\alpha}\rho\phi^5)| \leq Cs^2\phi^7 e^{2s\alpha} \rho 1_{\omega_0^4}.
$$
For the term in $\Delta w_3$,  estimate \eqref{eq:regularity} gives
\begin{align} \label{AX7}
s^5\int \! \! \! \int_{\omega_0^3 \times (0,T) } e^{2s\alpha}\phi^5  \Delta z_3 \Delta w_3   dxdt \leq C\|\rho F_3\|^2_{L^2(0,T;\bold{V})} +\delta  s^5 \int \! \! \! \int_Q e^{2s\alpha}\widehat{\phi}^{5} |\Delta z_3|^2dxdt.
\end{align}
Finally, for the last term we have
\begin{align} \label{AX8}
|s^5&\int \! \! \! \int_{\omega_0^3 \times (0,T) } e^{2s\alpha}\phi^5 \widehat{\phi}^{9/2} \Delta z_3  (\rho\widehat{\phi}^{-9/2})_t \Delta \varphi dxdt| \leq Cs^7|\int \! \! \! \int_{\omega_0^3 \times (0,T) } e^{2s\alpha}\widehat{\phi}^{6}\phi^5 \Delta z_3  (\Delta \psi +\Delta \eta)    dxdt| \nonumber \\
&\leq Cs^9\int \! \! \! \int_{\omega_0^3 \times (0,T)} e^{2s\alpha} \widehat{\phi}^{12}\phi^5|\Delta \psi |^2dxdt + C\|\rho F_3\|^2_{L^2(0,T;\bold{V})}+\delta  s^5 \int \! \! \! \int_Q e^{2s\alpha}\phi^{5} |\Delta z_3|^2dxdt,
\end{align}
because
$$|(\rho\widehat{\phi}^{-9/2})_t|  \leq Cs^{1+1/11} \widehat{\phi}^{-3}\rho.$$
\null
Thus, we have the following estimate for the local integral of $\Delta z_3$:
\begin{align}\label{AX10}
s^5\int \! \! \! \int_{\omega_0^3 \times (0,T) } & e^{2s\alpha}\phi^5|\Delta z_3|^2 dxdt \nonumber \\
\leq C &\biggl(s^9\int \! \! \! \int_{\omega_0^4 \times (0,T)}e^{2s\alpha} \widehat{\phi}^{18}|\Delta \psi |^2dxdt+ s^9\int \! \! \! \int_{\omega_0^4 \times (0,T) } e^{2s\alpha}\phi^9 |\rho|^2 |\xi|^2dxdt  +\|\rho F_3\|^2_{L^2(0,T;\bold{V})}\biggl) \nonumber \\
&+\delta  \bigl(s^5 \int \! \! \! \int_Q e^{2s\alpha}\widehat{\phi}^{5} |\Delta z_3|^2dxdt+  s^3\int \! \! \! \int_Q e^{2s\alpha}\phi^3| \nabla \Delta z_3|^2dx dt + \widehat{I}_{-2}(s, \nabla \nabla \Delta z_3)\bigl).
\end{align}


\textbf{Step 5.} \textit{Estimate of a local integral of $\Delta \psi$}.

 \null

In this step, we estimate the local integral of  $\Delta \psi$ in the right-hand side of \eqref{AX10}. For that, we use \eqref{adjoint-1} to write
\begin{align}\label{local}
& s^9\int \! \! \! \int_{\omega_0^4 \times (0,T)} e^{2s\alpha} \phi^{18}|\Delta \psi |^2dxdt  \\
& \leq \frac{1}{M}s^9\int \! \! \! \int_{\omega_0^5 \times (0,T)}\theta_5 e^{2s\alpha} \phi^{18}\widehat{\phi}^{-9/2}\rho\Delta \psi( \xi_t + \Delta \xi -M\xi +  f_2 - M\rho^{-1}\widehat{\phi}^{9/2}\Delta \eta)dxdt. \nonumber
\end{align}

The rest of this step is  devoted to estimate each  one of the terms in the right-hand side of the above integral. For the first term, we have the following estimate

\begin{claim}\label{claim1}
For any $\delta >0$, there exists $C>0$ such that
\begin{align}\label{M1X}
 |s^9\int \! \! \! \int_{\omega_0^5 \times (0,T)}&\theta_5 e^{2s\alpha} \phi^{18}\widehat{\phi}^{-9/2} \rho\Delta \psi \xi_t dxdt  |  \nonumber\\
 \leq &\  C\bigl(s^{33}\int \! \! \! \int_{\omega_0^6 \times (0,T)} e^{2s\alpha} \widehat{\phi}^{61}|\rho|^2|\xi|^2 dxdt  + \|\rho F_3\|^2_{L^2(0,T;\bold{V})}  \bigl) \nonumber \\
  & + \delta \bigl(I_2(s,\rho \widehat{\phi}^{-9/2} \xi)+  \widehat{I}_0(s,\Delta \psi)+ s^5\int \! \! \! \int_Q  e^{2s\alpha}\phi^5|\Delta z_3|^2dxdt\bigl).
\end{align}
\end{claim}
We prove Claim \ref{claim1} in appendix \ref{proofclaim1}.

Next, we integrate by  parts  the second term in   \eqref{local} to obtain
\begin{align}\label{klds}
s^9\int \! \! \! \int_{\omega_0^5 \times (0,T)}\theta_5 e^{2s\alpha}\rho \phi^{18}\widehat{\phi}^{-9/2}& \Delta \psi \Delta \xi  dxdt \nonumber \\
& \ \  \ \ =  -s^9\int \! \! \! \int_{\omega_0^5\times (0,T)}\Delta \psi\nabla (e^{2s\alpha}\rho \phi^{18}\widehat{\phi}^{-9/2}\theta_5) \cdot  \nabla \xi  dxdt \nonumber \\
 & \ \  \ \ - s^9\int \! \! \! \int_{\omega_0^5 \times (0,T)}\theta_5 e^{2s\alpha}\rho \phi^{18}\widehat{\phi}^{-9/2} \nabla (\Delta \psi) \cdot  \nabla \xi  dxdt. \nonumber  \\
 &\leq Cs^{17}\int \! \! \! \int_{\omega_0^5 \times (0,T)}e^{2s\alpha}\widehat{\phi}^{26}|\rho|^2|\nabla \xi|^2dxdt +\delta \widehat{I}_0(s,\Delta \psi).
\end{align}
because
$$
|\nabla (e^{2s\alpha}\rho \phi^{18}\widehat{\phi}^{-9/2}\theta_5)| \leq Cs \widehat{\phi}^{29/2} \rho e^{2s\alpha}1_{\omega_0^5}.
$$
Next,
\begin{align}\label{QX20}
s^{17}\int \! \! \! \int_{\omega_0^6 \times (0,T)}\theta_6e^{2s\alpha}\widehat{\phi}^{26}|\rho|^2|\nabla \xi|^2dxdt  & = -s^{17}\int \! \! \! \int_{\omega_0^6\times (0,T)}\theta_6 e^{2s\alpha}\widehat{\phi}^{26}|\rho|^2\Delta \xi \xi dxdt  \nonumber \\
&+ \frac{s^{17}}{2}\int \! \! \! \int_{\omega_0^6 \times (0,T)}\Delta (\theta_6 e^{2s\alpha}\widehat{\phi}^{26}|\rho|^2)  |\xi|^2 dxdt  \\
& \leq C s^{33}\int \! \! \! \int_{\omega_0^6 \times (0,T)} e^{2s\alpha} \widehat{\phi}^{61}|\rho|^2|\xi|^2 dxdt + \delta I_2(s,\rho \widehat{\phi}^{-9/2} \xi),\nonumber
\end{align}
because
$$
|\Delta (\theta_6 e^{2s\alpha}\widehat{\phi}^{26}|\rho|^2) | \leq Cs^2\widehat{\phi}^{28} |\rho|^2e^{2s\alpha}1_{\omega_0^6}.
$$

\null

Finally, for the last three terms, we have
\begin{align}\label{M14}
\left|s^9\int \! \! \! \int_{\omega_0^5 \times (0,T)}\theta_5 e^{2s\alpha} \phi^{18}\widehat{\phi}^{-9/2}\rho \Delta \psi \xi dxdt\right|  \leq  C  s^{15}\int \! \! \! \int_{\omega_0^5 \times (0,T)}\widehat{\phi}^{24}  e^{2s\alpha} |\rho|^2|\xi|^2dxdt + \delta \widehat{I}_0(s, \Delta \psi),
\end{align}
\begin{align}\label{M4}
|s^9\int \! \! \! \int_{\omega_0^5 \times (0,T)}\theta_5 &e^{2s\alpha}\rho \phi^{18}\widehat{\phi}^{-9/2} \Delta \psi f_2 dxdt | \leq  C  s^{15}\int \! \! \! \int_{\omega_0^5 \times (0,T)}\widehat{\phi}^{24}  e^{2s\alpha}|\rho|^2 |f_2|^2dxdt + \delta \widehat{I}_0(s, \Delta \psi)
\end{align}
and
\begin{align}\label{M5}
|s^9\int \! \! \! \int_{\omega_0^5 \times (0,T)}\theta_5 e^{2s\alpha}\phi^{18}\Delta \psi \Delta \eta dxdt|   \leq   C\int \! \! \! \int_Q\widehat{\phi}^{-9}|\rho|^2|f_1|^2dxdt + \delta \widehat{I}_0(s, \Delta \psi).
\end{align}
Gathering  \eqref{xm}, \eqref{local}-\eqref{M5}, we obtain,   after absorbing the lower order terms, the estimate:
\begin{align}\label{xk}
&   \widehat{I}_0(s, \Delta \psi)+I_2(s, \rho\widehat{\phi}^{-9/2}\xi) +\int \! \! \! \int_Q e^{2s\alpha}\widehat{\phi}^{-6}|\rho|^2|\Delta \varphi|^2dxdt +\int \! \! \! \int_Q e^{2s\widehat{\alpha}}\widehat{\phi}^{-6}|\rho|^2|\nabla \varphi|^2dxdt  \nonumber \\
&\leq  C\biggl(s^{33}\int \! \! \! \int_{\omega_0^6 \times (0,T)} e^{2s\alpha} \widehat{\phi}^{61}|\rho|^2|\xi|^2 dxdt  +  \int \! \! \! \int_Q  \widehat{\phi}^{-9}|\rho|^2|f_1|^2 dxdt   + s^{15}\int \! \! \! \int_{Q}\widehat{\phi}^{24}  e^{2s\alpha}|\rho|^2 |f_2|^2dxdt \nonumber \\
 & \ +  \int \! \! \! \int_Q e^{2s\alpha}\widehat{\phi}^{-9} |\rho|^2|\Delta v_3|^2dxdt +s\int \! \! \! \int_{\Sigma} e^{2s\alpha} \widehat{\phi}^{-8}|\rho|^2|\frac{\partial v_3}{\partial \nu}|^2d\sigma dt\biggl),
\end{align}
for  $C = C(\Omega, \omega)$ and every $s  \geq s_0(\Omega, \omega,T, \lambda) $. Notice that we can add the last term in the lef-hand side of \eqref{xk}  because $\frac{\partial \varphi}{\partial \nu }= 0$.

To finish the proof, we notice that
\begin{align}\label{BTQ1}
|\int \! \! \! \int_Q e^{2s\alpha}\widehat{\phi}^{-9}|\rho|^2 |\Delta v_3|^2dxdt | &+ s\int \! \! \! \int_{\Sigma} e^{2s\alpha}\widehat{\phi}^{-8}|\rho|^2|\frac{\partial v_3}{\partial \nu}|^2d\sigma dt \\
& \leq C s\int \! \! \! \int_Q e^{2s\alpha}\widehat{\phi}^{-8}|\Delta(z_3+w_3)|^2dxdt   \nonumber \\
&\leq   C\|\rho F_3\|^2_{L^2(0,T;\bold{V})} +\delta  s^5 \int \! \! \! \int_Q e^{2s\alpha}\widehat{\phi}^{5} |\Delta z_3|^2dxdt, \nonumber
\end{align}
for any $\delta >0$. Moreover, we also have

\begin{align}\label{VB1}
&|s^5\int \! \! \! \int_{\omega_0^3 \times (0,T) } e^{2s\alpha}\phi^5|\Delta z_1|^2 dxdt| \nonumber  \\
& \leq s^5|\int \! \! \! \int_{\omega_0^4 \times (0,T) } ( \Delta (\theta_4e^{2s\alpha}\phi^5)\Delta z_1  + 2\nabla(\theta_4e^{2s\alpha}\phi^5) \cdot  \nabla\Delta  z_1+  \theta_4e^{2s\alpha}\phi^5\Delta^2 z_1)z_1dxdt|  \nonumber \\
& \leq Cs^9\int \! \! \! \int_{\omega_0^4 \times (0,T) } e^{2s\alpha}\phi^9| z_1|^2 dxdt +\delta  \bigl(s^5 \int \! \! \! \int_Q e^{2s\alpha}\widehat{\phi}^{5} |\Delta z_1|^2dxdt \nonumber \\
&+  s^3\int \! \! \! \int_Q e^{2s\alpha}\phi^3| \nabla \Delta z_1|^2dx dt + \widehat{I}_{-2}(s, \nabla \nabla \Delta z_1)\bigl),
\end{align}
\null
since
$$
|\nabla(\theta_4e^{2s\alpha}\phi^5)| \leq Cs\phi^6  e^{2s\alpha}1_{\omega_0^4} \ \mbox{and} \  |\Delta (\theta_4e^{2s\alpha}\phi^5)| \leq C s^2\phi^7 e^{2s\alpha}1_{\omega_0^4}.
$$

From \eqref{Z11}, \eqref{AX10},  \eqref{xk}, \eqref{BTQ1} and \eqref{VB1},  we finish the proof of Theorem \ref{CarlemanP}.

\end{proof}

\section{Null controllability for the linear system}\label{sect3}

In this section we  solve the null controllability problem for the system \eqref{system-1}, with a right-hand side which decays exponentially as $t \rightarrow T^-$.


Indeed, we consider the system
\begin{equation}\label{system-1-L}
\left |
\begin{array}{ll}
\mathcal{L}(n,c, u) + (0,0, \nabla p)= (h_1,h_2+g_1\chi_{\omega_1}, H_3+g_2e_{N-2}\chi_{\omega_2}), \\
\nabla \cdot u =0 &     \mbox{in}  \  \ Q, \\
\frac{\partial n}{\partial \nu} = \frac{\partial c}{\partial \nu} = 0; \ u = 0    &    \mbox{on}  \  \   \Sigma, \\
n(x,0) = n_0;  \ c(x,0) = c_0; \ u(x,0) = u_0  &    \mbox{in}    \  \  \Omega,
\end{array}
\right.
\end{equation}
where
\begin{align}\label{system-1-L-1}
\mathcal{L}(n,c, u)& = \bigl(n_{t}    - \Delta n +M\Delta c , c_t   - \Delta c +Mc + M_0e^{-Mt}n, u_t - \Delta u -ne_N \bigl) \nonumber \\
&:= (\mathcal{L}_1, \mathcal{L}_2, \mathcal{L}_3)(n,c, u).
\end{align}
The aim is to  find $(g_1\chi_{\omega_1}, g_2\chi_{\omega_2})\in L^2(0,T; H^1(\Omega))  \times L^2(Q)$    ($g_2\equiv 0$, if $N=2$) such that the solution of  \eqref{system-1-L} satisfies
\be\label{system-1-L-2}
n(x,T) = c(x,T) = u(x,T) =0.
\ee
Furthermore, it will be necessary to solve \eqref{system-1-L} - \eqref{system-1-L-2}   in some appropriate weighted space. Before introducing such spaces, we improve the Carleman estimate given in Theorem \ref{CarlemanP}. This new Carleman inequality will only contain weight functions that do not vanish at $t=0$.

Let us consider a positive $C^{\infty}([0,T])$  function such that
\begin{equation}\label{system-linear-8}
\tilde{\ell}(t) = \left\{
\begin{array}{llcc}
\ell(T/2)  & \text{if} \ 0 \leq t \leq T/2    \\
\ell(t) & \text{if} \  3T/4 \leq t \leq T,
\end{array}
\right.
\end{equation}
and define our new weight functions as
$$
\beta(x,t) = \frac{e^{\lambda \eta_0(x)} - e^{2\lambda||\eta_0||_{\infty}}}{\tilde{\ell}(t)^{11}}, \ \gamma(x,t) = \frac{e^{\lambda \eta_0(x)}}{\tilde{\ell}(t)^{11}},
$$
\begin{align} \label{weightfunctions1-4}
\widehat{\gamma}(t) = \min_{x \in \overline{\Omega}} \gamma(x,t), \ \gamma^*(t) = \max_{x \in \overline{\Omega}} \phi(x,t), \ \beta^*(t) = \max_{x \in \overline{\Omega}}\beta (x,t),\ \widehat{\beta} = \min_{x \in \overline{\Omega}}\beta (x,t).
\end{align}
\\
With these new weights, we state our refined Carleman estimate as follows.

\begin{proposition}\label{CarlemanT=t}
Let  $(\varphi_T, \xi_T,v_T) \in L^2(\Omega)\times L^2(\Omega) \times \bold{H}$ and  $(f_1, f_2,F_3) \in L^2(Q)\times L^2(Q)\times L^2(0,T;\bold{V})$. There exists a positive constant $C$ depending on $T$, $s$ and $\lambda$, such that every solution of \eqref{adjoint-1} verifies:
\begin{align}\label{ME0}
&  \int_0^{T}\!\!\!\!\! \int_\Omega e^{5s\widehat{\beta}}\widehat{\gamma}^{-6} |\nabla \varphi|^2dxdt +\int_0^{T}\!\!\!\!\! \int_\Omega e^{5s\widehat{\beta}}\widehat{\gamma}^{-6} |\varphi-\bigl(\varphi\bigl)_\Omega |^2dxdt \nonumber \\
& + \int_0^{T}\!\!\!\!\! \int_\Omega e^{5s\widehat{\beta}}\widehat{\gamma}^{-4} | \xi |^2dxdt   + \int_0^{T}\!\!\!\!\! \int_\Omega e^{5s\widehat{\beta}}\widehat{\gamma}^{-6} | \nabla \xi |^2dxdt   \nonumber \\
&+\int_0^{T}\!\!\!\!\! \int_\Omega e^{5s\widehat{\beta}}\widehat{\gamma}^5|v|^2 dxdt + || \varphi (0)- \bigl(\varphi\bigl)_\Omega(0)||^2_{L^2(\Omega)}+ ||\xi(0)||^2_{L^2(\Omega)}+ ||v(0)||^2_{L^2(\Omega)} \nonumber \\
& \leq C \bigl( \int \! \! \! \int_{\omega_1 \times (0,T)} e^{2s\beta^*+3s\widehat{\beta}} \widehat{\gamma}^{61}|\chi_1|^2|\xi|^2 dxdt  +(N-2)\int \! \! \! \int_{\omega_2 \times (0,T) } e^{2s\beta^*+3s\widehat{\beta}}(\gamma^*)^9|\chi_2|^2| v_1|^2 dxdt\nonumber \\
&+  \int \! \! \! \int_Q  e^{3s\widehat{\beta}}\widehat{\gamma}^{-9}|f_1|^2 dxdt  + \int \! \! \! \int_{Q}\widehat{\gamma}^{24}  e^{2s\beta^*+3s\widehat{\beta}} |f_2|^2dxdt  +\int \! \! \! \int_{Q}e^{3 s\widehat{\beta}} (|F_3|^2 +|\nabla F_3|^2 )dxdt\bigl),
\end{align}

where
$$
 \bigl(\varphi\bigl)_\Omega(t) = \frac{1}{|\Omega|}\int_\Omega \varphi(x,t)dx.
$$

\end{proposition}

\begin{proof}

The proof of Proposition \ref{CarlemanT=t} is standard. It combines energy estimates and the Carleman inequality \eqref{Carlemanadjointsystem0}. For simplicity, we omit the proof.
\end{proof}

Now we proceed to the definition of the spaces where \eqref{system-1-L}-\eqref{system-1-L-2} will be solved. The main space will be:

\begin{align*}
E = \biggl \{&  (n,c,u,p, g_1, (N-2)g_2) \in E_0: \\
 & e^{-5/2s\widehat{\beta}}\widehat{\gamma}^{3} \mathcal{L}_1(n,c,u) \in L^2(Q),  e^{-5/2s\widehat{\beta}}\widehat{\gamma}^{2}  \biggl( \mathcal{L}_2(n,c,u) -g_1\chi_{1} \biggl ) \in  L^2(0,T; H^1(\Omega)),    \\
&  e^{-5/2s\widehat{\beta}}\widehat{\gamma}^{-5/2} \biggl(\mathcal{L}_3(n,c,u)+ \nabla p - e_{N-2}g_2\chi_{2}\biggl) \in \bold{L}^2(Q), \\
& \int_\Omega \mathcal{L}_1(n,c,u) dx = 0  \ \text{and} \ \frac{\partial n}{\partial \nu} = \frac{\partial c}{\partial \nu}= u =0  \ \text{on} \ \Sigma \biggl \},
\end{align*}
where
\begin{align*}
E_0 = \biggl\{ & (n,c,u,p, g_1,(N-2)g_2):  ||e^{-3/2s\widehat{\beta}}  \widehat{\gamma}^{9/2}n||_{L^2(Q)} + || e^{-s\beta^*-3/2s\widehat{\beta}}\widehat{\gamma}^{-12} c||_{L^2(Q)}  \nonumber \\
&+ || \chi_{1} e^{-s\beta^* - 3/2s\widehat{\beta}}\widehat{\gamma}^{-61/2}g_1||_{L^2(Q)} +(N-2) || \chi_{2} e^{-s\beta^* - 3/2s\widehat{\beta}}(\gamma^*)^{-9/2}g_2||_{L^2(Q)} \\
& + || e^{-3/2s\widehat{\beta}} u||_{L^2(0,T; \bold{H}^{-1}(\Omega))}  < \infty, \\
& e^{-5/4s\widehat{\beta}} \widehat{\gamma}^{13/4}n \in L^{2}(0,T; H^2(\Omega))\cap  L^{\infty}(0,T; H^1(\Omega)), \\
 &  e^{-5/4s\widehat{\beta}}\widehat{\gamma}^{-1/4} \nabla c \in L^2(0,T; \bold{H}^2(\Omega)), \\
 &e^{-3/2s\widehat{\beta}} \widehat{\gamma}^{-2-2/11} u \in L^2(0,T; \bold{H}^2(\Omega))\cap L^{\infty}(0,T; \bold{V})  \biggl\}.
\end{align*}
Notice that $E$ is a Banach space for the norm:
\begin{align}
||(n,c,u,&p, g_1,(N-2)g_2)||_E \nonumber \\
=&  ||e^{-3/2s\widehat{\beta}}  \widehat{\gamma}^{9/2}n||^2_{L^2(Q)} + || e^{-s\beta^*-3/2s\widehat{\beta}}\widehat{\gamma}^{-12} c||^2_{L^2(Q)}  \nonumber \\
&+ || \chi_{1} e^{-s\beta^* - 3/2s\widehat{\beta}}\widehat{\gamma}^{-61/2}g_1||^2_{L^2(Q)} +(N-2) || \chi_{2} e^{-s\beta^* - 3/2s\widehat{\beta}}(\gamma^*)^{-9/2}g_2||^2_{L^2(Q)} \nonumber \\
& + || e^{-3/2s\widehat{\beta}} u||^2_{L^2(0,T; \bold{H}^{-1}(\Omega))}  \nonumber \\
& +  ||e^{-5/2s\widehat{\beta}}\widehat{\gamma}^{3} \mathcal{L}_1(n,c,u)||^2_{L^2(Q)} + ||e^{-5/2s\widehat{\beta}}\widehat{\gamma}^{2}  \biggl(  \mathcal{L}_2(n,c,u) -g_1\chi_{1} \biggl )||^2_{L^2(0,T; H^1(\Omega))} \nonumber \\
&+ || e^{-5/2s\widehat{\beta}}\widehat{\gamma}^{-5/2} \biggl( \mathcal{L}_3(n,c,u)+ \nabla p - e_{N-2}g_2\chi_{2}\biggl)||^2_{\bold{L}^2(Q)} \nonumber \\
&+  ||e^{-5/4s\widehat{\beta}} \widehat{\gamma}^{13/4}n||^2_{L^{2}(0,T; H^2(\Omega))}+ ||e^{-5/4s\widehat{\beta}} \widehat{\gamma}^{13/4}n||^2_{L^{\infty}(0,T; H^1(\Omega))} \nonumber \\
& +  ||e^{-5/4s\widehat{\beta}}\widehat{\gamma}^{-1/4} \nabla c||^2_{L^2(0,T; \bold{H}^2(\Omega))} \nonumber \\
& +||e^{-3/2s\widehat{\beta}} \widehat{\gamma}^{-2-1/11} u||^2_{L^2(0,T; \bold{H}^2(\Omega))}+||e^{-3/2s\widehat{\beta}} \widehat{\gamma}^{-2-2/11} u||^2_{L^{\infty}(0,T; \bold{V})}.
\end{align}

\begin{remark}
For every  $(n,c,u,p, g_1,(N-2)g_2)\in E_0$, we have that $\nabla \cdot (n \nabla c) \in L^2(e^{-5s\widehat{\beta}}\widehat{\gamma}^{6}; Q)$. In fact,
\begin{align*}
 &\int \! \! \! \int_Q e^{-5s\widehat{\beta}}\widehat{\gamma}^{6}|\nabla \cdot (n \nabla c)|^2dxdt  \leq   \int \! \! \! \int_Q e^{-5s\widehat{\beta}}\widehat{\gamma}^{6}( | \nabla n|^2 |\nabla c|^2 + |n|^2 |\Delta c|^2)dxdt \\
&  \leq   \int \! \! \! \int_Q  (|e^{-5/4s\widehat{\beta}} \widehat{\gamma}^{13/4} \nabla n|^2 |e^{-5/4s\widehat{\beta}}\widehat{\gamma}^{-1/4}  \nabla c|^2 + |e^{-5/4s\widehat{\beta}} \widehat{\gamma}^{13/4}n|^2 |e^{-5/4s\widehat{\beta}}\widehat{\gamma}^{-1/4}  \Delta c|^2)dxdt <\infty.
\end{align*}
\end{remark}

\begin{remark}
If $(n,c,u,p, g_1,(N-2)g_2) \in E$,  then $n(T)=c(T) = u(T)=0$, so that $(n,c,u,p, g_1,(N-2)g_2)$ solve a null controllability problem for system \eqref{system-1-L} with an appropriate right-hand side $(h_1,h_2, H_3)$.
\end{remark}

We have the following result:
\begin{proposition}\label{UnifResult} Assume that:
\begin{align}
(n_0,c_0, u_0)\in H^1(\Omega)\times H^2(\Omega)\times \bold{V}, \ \int_\Omega n_0dx=0, \ \frac{\partial c_0}{\partial \nu} =0 \ \text{on} \ \partial \Omega,
\end{align}
$$
e^{-5/2s\widehat{\beta}}\widehat{\gamma}^{3} h_1 \in L^2(0,T; L^2_0(\Omega)),e^{-5/2s\widehat{\beta}}\widehat{\gamma}^{2}h_2 \in L^2(0,T; H^1(\Omega)),
$$
and
$$
e^{-5/2s\widehat{\beta}}\widehat{\gamma}^{-5/2}H_3 \in \bold{L}^2(Q).
$$
Then, there exist  $(g_1\chi_1, (N-2)g_2\chi_2) \in L^2( 0,T; H^1(\Omega))\times L^2(Q)$,  such that, if $(n,c,u,p)$ is the associated solution to \eqref{system-1-L}, one has $ (n,c,u,p, g_1\chi_1,(N-2)g_2\chi_2) \in E$. In particular,  \eqref{system-1-L-2} holds.
\end{proposition}
\begin{proof}
Following the  arguments in \cite{F-Im, Im-1},  we introduce the space
%
%
 \begin{align*}
P_0 = \biggl\{ (z,w, y, q) & \in \bold{C}^{3}(\overline{Q});  \ \frac{\partial z}{\partial \nu} = \frac{\partial w}{\partial \nu}  =y= 0  \ \mbox{on} \ \Sigma , \int_\Omega z(x,T)dx = 0, \nonumber \\
& \nabla \cdot y =0, \ \int_{\Omega} q(x,t)dx =0, \Delta q=0,  \bigl(\mathcal{L}^*_3(z, w, y) + \nabla q\bigl)\bigl|_{\Sigma}=0  \biggl \}
 \end{align*}
and consider the bilinear form on $P_0$:
\begin{align}
a&\bigl((\widehat{z},\widehat{w}, \widehat{y}, \widehat{q}), (z,w, y, q) \bigl) \nonumber \\
&:=  \int \! \! \! \int_Q  e^{3s\widehat{\beta}}\widehat{\gamma}^{-9} \mathcal{L}^*_1(\widehat{z}, \widehat{w}, \widehat{y})\mathcal{L}^*_1(z, w, y) dxdt  + \int \! \! \! \int_{Q}\widehat{\gamma}^{24}  e^{2s\beta^*+3s\widehat{\beta}} \mathcal{L}^*_2(\widehat{z}, \widehat{w}, \widehat{y})\mathcal{L}^*_2(z, w, y)dxdt \nonumber \\
& +\int \! \! \! \int_{Q}e^{3 s\widehat{\beta}} \biggl(\bigl[ \mathcal{L}^*_3(\widehat{z}, \widehat{w}, \widehat{y}) +\nabla \widehat{q}\bigl] \cdot \bigl[ \mathcal{L}^*_3(z, w, y) +\nabla q \bigl] +\nabla \bigl[ \mathcal{L}^*_3(\widehat{z}, \widehat{w}, \widehat{y}) +\nabla \widehat{q}\bigl] :\nabla \bigl[ \mathcal{L}^*_3(z, w, y) +\nabla q \bigl]   \biggl)dxdt \nonumber \\
& +\int \! \! \! \int_{\omega_1 \times (0,T)} e^{2s\beta^*+3s\widehat{\beta}} \widehat{\gamma}^{61}|\chi_1|^2\widehat{w}w dxdt  +(N-2)\int \! \! \! \int_{\omega_2 \times (0,T) } e^{2s\beta^*+3s\widehat{\beta}}(\gamma^*)^9|\chi_2|^2\widehat{y}_1y_1 dxdt.
\end{align}
Here, we have denoted  $\mathcal{L}^*$ is the adjoint of  $\mathcal{L}$, i.e.,
\begin{align*}
\mathcal{L}^*(z, w, y) &= \bigl( -z_{t}    - \Delta z+M_0e^{-Mt}w -  ye_N, -w_t   - \Delta w +Mw + M \Delta z, -y_t - \Delta y \bigl) \nonumber \\
&:= (\mathcal{L}^*_1, \mathcal{L}^*_2, \mathcal{L}^*_3)(z, w, y).
\end{align*}

Thanks to \eqref{ME0}, we have that $a: P_0\times P_0 \rightarrow \mathbb{R}$ is a symmetric, definite positive bilinear form. We denote by $P$ the completion of $P_0$ with respect to the norm associated to $a(.,.)$ (which we denote by $||.||_P$). This is a Hilbert space and $a(.,.)$ is a continuous and coercive bilinear form on $P$.

Let us now consider the linear form
\begin{align*}
\bigl<G, &(z,w, y, q)\bigl> \nonumber \\
&=  \int \! \! \! \int_Q h_1z dxdt + \int \! \! \! \int_Q h_2wdxdt + \int_0^T H_3 \cdot ydxdt + \int_{\Omega}\bigl(n_0z(0) + c_0w(0) + u_0\cdot y(0) \bigl)dx.
\end{align*}
It is immediate to see that
\begin{align*}
| \bigl<G, (z,w, y, q)\bigl>| =  & \|e^{-5/2s\widehat{\beta}}\widehat{\gamma}^{3} h_1\|_{L^2(0,T; L^2_0(\Omega))} \|e^{5/2s\widehat{\beta}}\widehat{\gamma}^{-3} \bigl(z-\bigl(z\bigl)_{\Omega}\bigl)\|_{L^2(Q)}\nonumber \\
&   +\|e^{-5/2s\widehat{\beta}}\widehat{\gamma}^{2}h_2\|_{L^2(Q)}\|e^{5/2s\widehat{\beta}}\widehat{\gamma}^{-2}w\|_{L^2(Q)} \nonumber \\
&  +\|e^{-5/2s\widehat{\beta}}\widehat{\gamma}^{-5/2}H_3\|_{\bold{L}^2(Q)}\|e^{5/2s\widehat{\beta}}\widehat{\gamma}^{5/2}y\|_{\bold{L}^2(Q)}\nonumber \\
&  + \|(n_0, c_0, u_0)\|_{\bold{L}^2(\Omega)}\|(z(0)-\bigl(z\bigl)_{\Omega}(0), w(0), y(0)) \|_{\bold{L}^2(\Omega)}.
\end{align*}
In particular, we have  that (see \eqref{ME0})
\begin{align*}
| \bigl<G, (z,w, y, q)\bigl>| \leq & \ C \biggl( \|e^{-5/2s\widehat{\beta}}\widehat{\gamma}^{3} h_1\|_{L^2(0,T; L^2_0(\Omega))} +\|e^{-5/2s\widehat{\beta}}\widehat{\gamma}^{2}h_2\|_{L^2(Q)} \nonumber \\
& +\|e^{-5/2s\widehat{\beta}}\widehat{\gamma}^{-5/2}H_3\|_{\bold{L}^2(Q)} + \|(n_0, c_0, u_0)\|_{\bold{L}^2(\Omega)}\biggl)\| (z,w, y, q)\|_P.
\end{align*}
Therefore, $G$ is a linear form on $P$ and by Lax-Milgram's lemma, there exists a unique $(\widehat{z},\widehat{w}, \widehat{y}, \widehat{q}) \in P$ such that
\be\label{LAxM}
a\bigl((\widehat{z},\widehat{w}, \widehat{y}, \widehat{q}), (z,w, y, q) \bigl)= \bigl<G, (z,w, y, q)\bigl>,
\ee
for every $(z,w, y, q)  \in P$. We set
\begin{align}\label{extremalsol}
& (\widehat{n},\widehat{c},\widehat{u})  \nonumber \\
&=(e^{3s\widehat{\beta}}  \widehat{\gamma}^{-9}\mathcal{L}^*_1(\widehat{z},\widehat{w}, \widehat{y}), e^{2s\beta^*+3s\widehat{\beta}}\widehat{\gamma}^{24}\mathcal{L}^*_2(\widehat{z},\widehat{w}, \widehat{y}), e^{3s\widehat{\beta}} (\mathcal{L}^*_3(\widehat{z},\widehat{w}, \widehat{y}) +\nabla \widehat{q}-\Delta(\mathcal{L}^*_3(\widehat{z},\widehat{w}, \widehat{y}) +\nabla \widehat{q})\bigl)
\end{align}
and
\begin{align}\label{DefcontrolX}
(\widehat{g}_1,(N-2)\widehat{g}_2)= -( e^{2s\beta^* +3s\widehat{\beta}}\widehat{\gamma}^{61} \widehat{w}\chi_1, (N-2)e^{2s\beta^* + 3s\widehat{\beta}}(\gamma^*)^{9}y_1\chi_2).
\end{align}

Let us show that  the quantity
\begin{align*}
&||e^{-3/2s\widehat{\beta}}  \widehat{\gamma}^{9/2}\widehat{n}||^2_{L^2(Q)} + || e^{-s\beta^*-3/2s\widehat{\beta}}\widehat{\gamma}^{-12}\widehat{c}||^2_{L^2(Q)}+ || e^{-3/2s\widehat{\beta}} \widehat{u}||^2_{L^2(0,T; \bold{H}^{-1}(\Omega))}  \nonumber \\
&+ || \chi_{1} e^{-s\beta^* - 3/2s\widehat{\beta}}\widehat{\gamma}^{-61/2}\widehat{g}_1||^2_{L^2(Q)} +(N-2) || \chi_{2} e^{-s\beta^* - 3/2s\widehat{\beta}}(\gamma^*)^{-9/2}\widehat{g}_2||^2_{L^2(Q)}
\end{align*}
is finite.

We begin noticing that
\begin{align*}
\int_0^T e^{-3s\widehat{\beta}} ||\widehat{u}||^2_{\bold{H}^{-1}(\Omega)}dt  &=  \int_0^T e^{-3s\widehat{\beta}} \sup_{||\zeta||_{\bold{H}^{1}_0(\Omega)}=1}<\widehat{u}, \zeta>^2_{\bold{H}^{-1}(\Omega), \bold{H}^{1}_0(\Omega)}dt \nonumber \\
&= \int_0^T e^{3s\widehat{\beta}} \sup_{||\zeta||_{\bold{H}^{1}_0(\Omega)}=1}<  \mathcal{L}^*_3(\widehat{z},\widehat{w}, \widehat{y}) +\nabla \widehat{q}-\Delta(\mathcal{L}^*_3(\widehat{z},\widehat{w}, \widehat{y}) +\nabla \widehat{q}), \zeta>^2_{\bold{H}^{-1}(\Omega), \bold{H}^{1}_0(\Omega)}dt \nonumber \\
& = \int_0^T \sup_{||\zeta||_{\bold{H}^{1}_0(\Omega)}=1}\bigl(e^{3/2s\widehat{\beta}} (\mathcal{L}^*_3(\widehat{z},\widehat{w}, \widehat{y}) +\nabla \widehat{q}), \zeta\bigl)^2_{\bold{L}^{2}(\Omega)}+\bigl(e^{3/2s\widehat{\beta}} \nabla (\mathcal{L}^*_3(\widehat{z},\widehat{w}, \widehat{y}) +\nabla \widehat{q}), \nabla \zeta\bigl)^2_{\bold{L}^{2}(\Omega)} dt \nonumber \\
&\leq   \int \! \! \! \int_Q  e^{3s\widehat{\beta}}(|\mathcal{L}^*_3(\widehat{z},\widehat{w}, \widehat{y}) +\nabla \widehat{q}|^2+ |\nabla (\mathcal{L}^*_3(\widehat{z},\widehat{w}, \widehat{y}) +\nabla \widehat{q})|^2)dxdt.
\end{align*}
Moreover,  since $\nabla \cdot y =0, \Delta q=0$ and $\bigl(\mathcal{L}^*_3(z, w, y)+ \nabla q\bigl)\bigl|_{\Sigma}=0$, we have that  $e^{3/2s\widehat{\beta}}\bigl(\mathcal{L}^*_3(z, w, y)+ \nabla q\bigl) \in L^2(0,T;\bold{V})$ and the equality is achieved. It is now immediate to see that
\begin{align}
&||e^{-3/2s\widehat{\beta}}  \widehat{\gamma}^{9/2}\widehat{n}||^2_{L^2(Q)} + || e^{-s\beta^*-3/2s\widehat{\beta}}\widehat{\gamma}^{-12}\widehat{c}||^2_{L^2(Q)}+ || e^{-3/2s\widehat{\beta}} \widehat{u}||^2_{L^2(0,T; \bold{H}^{-1}(\Omega))}  \nonumber \\
&+ || \chi_{1} e^{-s\beta^* - 3/2s\widehat{\beta}}\widehat{\gamma}^{-61/2}\widehat{g}_1||^2_{L^2(Q)} +(N-2) || \chi_{2} e^{-s\beta^* - 3/2s\widehat{\beta}}(\gamma^*)^{-9/2}\widehat{g}_2||^2_{L^2(Q)} \nonumber \\
&= a \bigl( (\widehat{n},\widehat{c},\widehat{u}, \widehat{q}),  (\widehat{n},\widehat{c},\widehat{u}, \widehat{q}) \bigl)  <\infty. \label{finitenormaextremsol}
\end{align}

Let us show that,  $(\widehat{n},\widehat{c},\widehat{u})$ is the weak solution of \eqref{system-1-L} with $(g_1,g_2)= ( \widehat{g}_1,  \widehat{g}_2)$.

First, it is not difficult to see that the weak solution   $(\tilde{n}, \tilde{c}, \tilde{u})$ of system \eqref{system-1-L} with $g_1= \widehat{g}_1$ and $g_2= \widehat{g}_2$ satisfies the following identity
\begin{align}\label{deftrans}
 &\int \! \! \! \int_Q  (\tilde{n}, \tilde{c}, \tilde{u})\cdot (f_1,f_2) dxdt +\int_0^T< \tilde{u}, F_3>_{\bold{H}^{-1}(\Omega), \bold{H}^{1}_0(\Omega)} dxdt  \nonumber \\
 &= \int \! \! \! \int_Q h_1\varphi dxdt +\int \! \! \! \int_Q h_2\xi dxdt+\int \! \! \! \int_Q H_3\cdot  vdxdt \nonumber \\
 &+\int \! \! \! \int_Q \widehat{g}_1\chi_1 \xi dxdt +(N-2)\int \! \! \! \int_Q \widehat{g}_2\chi_2 v_1 dxdt  \nonumber \\
&+ (n_0,\varphi(0)) +(c_0,w(0)) +(u_0,v(0)), \forall (f_1,f_2,F_3) \in L^2(Q)^2\times L^2(0,T;\bold{V}),
\end{align}
where $(\varphi, \xi, v, \pi)$ is the solution of
\begin{equation}\label{system-1-L-adjoint}
\left |
\begin{array}{ll}
\mathcal{L}^*(\varphi, \xi, v) + (0,0, \nabla \pi)= (f_1,f_2, F_3) &     \mbox{in}  \  \ Q, \\
\nabla \cdot v =0 &     \mbox{in}  \  \ Q, \\
\frac{\partial \varphi}{\partial \nu} = \frac{\partial \xi }{\partial \nu} = 0; \ v = 0    &    \mbox{on}  \  \   \Sigma, \\
\varphi(x,T) = 0;  \ \xi(x,T) = 0; \ v(x,T) = 0  &    \mbox{in}    \  \  \Omega.
\end{array}
\right.
\end{equation}

Let us now take $(f_1^k,f_2^k, F_3^k) \in C^{\infty}_0(Q) \times C^{\infty}_0(Q) \times C^{\infty}_0(0,T; \mathcal{V})$ converging to $(f_1,f_2, F_3)$ as $k \rightarrow \infty$. Here  $\mathcal{V} = \{ u \in  \bold{C}^{\infty}_0(\Omega), \ \nabla \cdot u=0 \ \mbox{in} \ \Omega \}$. Moreover, let $(\varphi^k, \xi^k, v^k, \pi^k)$ be the solution of
\begin{equation}\label{system-1-L-adjoint-X}
\left |
\begin{array}{ll}
\mathcal{L}^*(\varphi^k, \xi^k, v^k) + (0,0, \nabla \pi^k)= (f_1^k, f_2^k, F_3^k) &     \mbox{in}  \  \ Q, \\
\nabla \cdot v^k =0 &     \mbox{in}  \  \ Q, \\
\frac{\partial \varphi^k}{\partial \nu} = \frac{\partial \xi^k }{\partial \nu} = 0; \ v^k = 0    &    \mbox{on}  \  \   \Sigma, \\
\varphi^k(x,T) = 0;  \ \xi^k(x,T) = 0; \ v^k(x,T) = 0  &    \mbox{in}    \  \  \Omega.
\end{array}
\right.
\end{equation}	
We have that $(\varphi^k, \xi^k, v^k, \pi^k) \in P_0$ and from energy estimates, we have that   $(\varphi^k, \xi^k, v^k)$ converges to $(\varphi, \xi, v, \pi)$ in the space $L^2(Q)\times L^2(Q) \times L^2(0,T;\bold{V})$ (actually it converges in a better space).

From \eqref{LAxM} and the definition of $(\widehat{n},\widehat{c},\widehat{u})$, we have
\begin{align}\label{PX1}
&  \int \! \! \! \int_Q  \widehat{n}f_1^kdxdt  + \int \! \! \! \int_{Q}\widehat{c}f_2^kdxdt +\int_0^T <\widehat{u}, F_3^k>_{\bold{H}^{-1}(\Omega), \bold{H}^{1}_0(\Omega)} dxdt \nonumber \\
& = \int \! \! \! \int_Q h_1 \varphi^k dxdt + \int \! \! \! \int_Q h_2 \xi^k dxdt + \int_0^T H_3 \cdot v^k dxdt + \int_{\Omega}\bigl(n_0\varphi^k(0) + c_0\xi^k(0) + u_0\cdot v^k(0) \bigl)dx \nonumber \\
&+\int \! \! \! \int_{\omega_1 \times (0,T)} \chi_1 \widehat{g}_1\xi^k dxdt  +(N-2)\int \! \! \! \int_{\omega_2 \times (0,T) } \chi_2\widehat{g}_2v_1^k dxdt.
\end{align}
We may pass to the limit in \eqref{PX1} to conclude that $(\widehat{n},\widehat{c},\widehat{u})$ also satisfies \eqref{deftrans} for every $(f_1,f_2, F_3) \in L^2(Q)\times L^2(Q) \times L^2(0,T;\bold{V})$.

The following lemma says that, possibly changing $\widehat{q}$ in \eqref{extremalsol},  $(\widehat{n},\widehat{c},\widehat{u})$ is in fact the weak solution of \eqref{system-1-L}.

\begin{lemma}\label{MN1}
Let $u \in  L^2(0,T;\bold{H}^{-1}(\Omega)$ with $\nabla \cdot u = 0$ and such that
$$
\int_0^T <u,F>_{\bold{H}^{-1}(\Omega), \bold{H}^{1}_0(\Omega)} dt =0
$$
for every $F \in  L^2(0,T;\bold{V})$. Then there exist $q \in L^2(0,T; L^2_0(\Omega))$, with $\Delta q=0$,  such that
$$
u= \nabla q.
$$
\end{lemma}
\begin{proof}
The result follows  from de Rham's theorem.
\end{proof}

From Lemma \ref{MN1}, identities \eqref{deftrans} and   \eqref{PX1},  we conclude that $(\widehat{n},\widehat{c},\widehat{u})$ is in fact the weak solution of \eqref{system-1-L}.



Let us now show that $(\widehat{n},\widehat{c},\widehat{u})$ belongs to $E$. Indeed, it only remains to check that
\begin{align*}
e^{-5/4s\widehat{\beta}} \widehat{\gamma}^{13/4} \widehat{n}  \in L^{2}(0,T; H^2(\Omega)) \cap  L^{\infty}(0,T; H^1(\Omega)),
\end{align*}
\begin{align*}
  e^{-5/4s\widehat{\beta}} \widehat{\gamma}^{-1/4}  \hat{c} \in L^2(0,T; H^3(\Omega))
\end{align*}
and that
\begin{align*}
e^{-3/2s\widehat{\beta}} \widehat{\gamma}^{-2-2/11} \widehat{u} \in L^2(0,T; \bold{H}^2(\Omega))\cap L^{\infty}(0,T; \bold{V}).
\end{align*}
To this end, let us introduce   $(n^*,c^*, u^*) = \rho(t)(\widehat{n},\widehat{c}, \widehat{u})$, which satisfies

\begin{equation}\label{system-1-regular}
\left |
\begin{array}{ll}
n_{t}^*    - \Delta n^*  = - M \Delta c^* + \rho h_1  +\rho_t \widehat{n}  &     \mbox{in}  \  \ Q,  \\
c_t^*   - \Delta c^*  = -Mc^*  -M_0e^{-Mt}n^* + \rho g_1 \chi_{1} + \rho h_2 +\rho_t \widehat{c}   &     \mbox{in}  \  \ Q, \\
u_t^* - \Delta u^*  + \nabla p^* = n^*e_N + \rho g_2\chi_{2}e_{N-2}  +\rho H_3 + \rho_t \widehat{u}&     \mbox{in}  \  \ Q,  \\
\nabla \cdot u^* =0 &     \mbox{in}  \  \ Q, \\
\frac{\partial n^*}{\partial \nu} = \frac{\partial c^*}{\partial \nu} = 0; \ u^* = 0    &    \mbox{on}  \  \   \Sigma, \\
n^*(x,0) = \rho(0)n_0;  \ c^*(x,0) = \rho(0)c_0; \ u^*(x,0) = \rho(0)u_0  &    \mbox{in}    \  \  \Omega,
\end{array}
\right.
\end{equation}

We  consider four cases:

\null

\textit{Case 1}. $\rho = e^{-5/4s\widehat{\beta}} \widehat{\gamma}^{13/4}$.

\null

In this case, we have that
\begin{align}
|\rho_{t}| \leq  Ce^{-5/4s\widehat{\beta}} \widehat{\gamma}^{9/2} \leq Ce^{-3/2s\widehat{\beta}}  \widehat{\gamma}^{9/2}
\end{align}
and
\begin{align}
|\rho_{t}| \leq  Ce^{-s\beta^*-3/2s\widehat{\beta}}\widehat{\gamma}^{-12}.
\end{align}
From \eqref{finitenormaextremsol}, it follows that  $\rho_t \hat{n}$  and $\rho_{t} \hat{c}$ belong to $L^2(Q)$. Therefore, from well-known regularity properties of  parabolic systems (see, for instance, \cite{L-S-U}), we have
\begin{equation}
\begin{cases}
e^{-5/4s\widehat{\beta}} \widehat{\gamma}^{13/4} \hat{n} \in L^{2}(0,T;H^2(\Omega)) \cap  L^\infty(0,T; H^1(\Omega)), \\
e^{-5/4s\widehat{\beta}} \widehat{\gamma}^{13/4} \hat{c} \in L^2(0,T;H^2(\Omega)).
\end{cases}
\end{equation}

\null

\textit{Case 2}. $\rho = e^{-5/4s\widehat{\beta}} \widehat{\gamma}^{-1/4}$.

\null

In this case, a simple calculation gives
\begin{align}
|\rho_{t}|  \leq Ce^{-5/4s\widehat{\beta}} \widehat{\gamma}^{13/4}
\end{align}
 and from  Case $1$, we conclude that  $\rho_t \widehat{c}$ belongs to $L^2(0,T; H^1(\Omega))$.

Using the definition of $\widehat{g}_1$ (see \eqref{DefcontrolX}) and \eqref{ME0}, we can also show that
\begin{align}
 \iint\limits_Q |\nabla (e^{-5/4s\widehat{\beta}} \widehat{\gamma}^{-1/4} \widehat{g}_1)|^2 \leq Ca((\hat{z}, \hat{w}),(\hat{z}, \hat{w})),
\end{align}
for some $C>0$, since $e^{7/2s\widehat{\beta}+4s\beta^*}\widehat{\gamma}^{122-1/2}\leq Ce^{5s\widehat{\beta}}\widehat{\gamma}^{-6}$. Hence it follows that  $e^{-5/4s\widehat{\beta}} \widehat{\gamma}^{-1/4} \widehat{g} \in L^2(0,T; H^1(\Omega))$.

Therefore,  from the regularity theory  for parabolic systems,   we deduce that
\begin{equation}
\begin{cases}
e^{-5/4s\widehat{\beta}} \widehat{\gamma}^{-1/4}  \widehat{n} \in L^{\infty}(0,T;H^1(\Omega)) \cap L^2(0,T;H^2(\Omega)), \\
e^{-5/4s\widehat{\beta}} \widehat{\gamma}^{-1/4}  \widehat{c} \in L^2(0,T;H^3(\Omega)).
\end{cases}
\end{equation}

\null

\textit{Case 3}. $\rho =e^{-3/2s\widehat{\beta}} \widehat{\gamma}^{-1-1/11}$.

\null

In this case, we have
\begin{align}
|\rho_{t}|  \leq Ce^{-3/2s\widehat{\beta}}.
\end{align}
and it follows that
\begin{align*}
e^{-3/2s\widehat{\beta}}  \widehat{\gamma}^{-1-1/11} \widehat{u} \in L^2(0,T; \bold{H}^1(\Omega))\cap L^{\infty}(0,T; \bold{H}).
\end{align*}

\textit{Case 4}. $\rho =e^{-3/2s\widehat{\beta}} \widehat{\gamma}^{-2-2/11}$.

\null

In this case, we have
\begin{align}
|\rho_{t}|  \leq Ce^{-3/2s\widehat{\beta}} \widehat{\gamma}^{-1-1/11} .
\end{align}
and it follows that
\begin{align*}
e^{-3/2s\widehat{\beta}} \widehat{\gamma}^{-2-2/11} \widehat{u} \in L^2(0,T; \bold{H}^2(\Omega))\cap L^{\infty}(0,T; \bold{V}).
\end{align*}

This finishes the proof of Proposition \ref{UnifResult}.

\begin{remark}
For every $a>0$, and every $b, c\in \mathbb{R}$, the function $s^be^{as\widehat{\beta}}\gamma^{c}$ is bounded.
\end{remark}

%
%

\end{proof}

\section{Null controllability to trajectories}\label{sec4}

In this section we give the proof of Theorem \ref{mainresultt} using similar arguments to those employed, for instance, in \cite{Im-1}. We will see that the results obtained in the previous section allow us to locally invert the nonlinear system \eqref{system}. In fact, the regularity deduced for the solution of the linearized system \eqref{system-1} will be sufficient to apply a suitable inverse function theorem (see Theorem \ref{teoremadecontrol}  below).
Thus, let us set $n = M + z$, $c = M_0e^{-Mt}+ w$ and $u= y$   and let us use these equalities in \eqref{system}. We find:
\begin{equation}\label{IFS}
\left |
\begin{array}{ll}
\mathcal{L}(z,w, y) + (0,0, \nabla p)=  -( y\cdot \nabla z +\nabla \cdot (z\nabla w), zw+y\cdot \nabla w, (y\cdot \nabla) y)  + ( 0, g_1 \chi_1, (N-2)g_2\chi_2), \\
\nabla \cdot y =0 &     \mbox{in}  \  \ Q, \\
\frac{\partial z}{\partial \nu} = \frac{\partial w}{\partial \nu} = 0; \ y = 0    &    \mbox{on}  \  \   \Sigma, \\
z(x,0) = n_0-M;  \ w(x,0) = c_0 -M_0; \ y(x,0) = u_0  &    \mbox{in}    \  \  \Omega,
\end{array}
\right.
\end{equation}

This way, we have reduced our problem to a local null controllability result for the solution  $(z,w,y)$ of the nonlinear problem \eqref{IFS}. We will use the following inverse mapping theorem (see \cite{Graves}):


\begin{theorem} \label{teoremadecontrol}
Let $E$ and $G$ be two Banach spaces and let  $ \mathcal{A} : E \rightarrow G $ be a continuous function from $E$ to $G$ defined in $B_{\eta}(0)$ for some $\eta >0$ with $\mathcal{A}(0) =0$. Let $\Lambda$ be a continuous and linear operator from $E$ onto $G$ and suppose there exists $K_0 >0$ such that
\be\label{estimavt}
||e||_E \leq K_0||\Lambda (e)||_G
\ee
and that there exists $\delta < K_0^{-1}$ such that
  \be\label{strictdiffer-1}
|| \mathcal{A} (e_1) - \mathcal{A} (e_2) -\Lambda(e_1-e_2)|| \leq \delta ||e_1 -e_2||
\ee
whenever $e_1,e_2 \in B_{\eta}(0)$. Then the equation $\mathcal{A}(e) = h $ has a solution $e \in B_{\eta}(0)$ whenever $||h||_G \leq c \eta$, where $c =  K_0^{-1}-\delta$.
\end{theorem}

\begin{remark}
In the case where $ \mathcal{A} \in C^1(E;G)$,using the mean value theorem,  it can be shown, that for any $\delta < K_0^{-1}$, inequality \eqref{strictdiffer-1} is satisfied with $\Lambda = \mathcal{A}'(0)$ and $\eta >0$ the continuity constant at zero, i. e.,
\be\label{constantuniform1}
||\mathcal{A}'(e) -\mathcal{A}'(0)||_{\mathcal{L}(E;G)} \leq \delta
\ee
whenever $||e||_E \leq \eta$.
\end{remark}

%
%
%
%
%
%
In our setting, we use this theorem with the space $E$ and
$$
G = X  \times Y,
$$
where

\begin{align}\label{X}
X= \{ (h_1, h_2, H_3); & \   e^{-5/2s\widehat{\beta}}\widehat{\gamma}^{3}h_1 \in L^2(Q),  e^{-5/2s\widehat{\beta}}\widehat{\gamma}^{2}h_2 \in L^2(0,T;H^1(\Omega),  \\
& e^{-5/2s\widehat{\beta}}\widehat{\gamma}^{-5/2}H_3 \in \bold{L}^2(Q) \  \text{and} \  \int_\Omega h_1(x,t) dx = 0 \ \text{a. e.}  \ t \in (0,T)  \}, \nonumber
\end{align}
\begin{align}
Y = \{ (z_0,w_0, y_0 ) \in H^1(\Omega) \times H^2(\Omega)\times \bold{V}; \  \int_\Omega z_0dx=0,  \ \frac{\partial w_0}{\partial \nu} =0 \ \text{on} \ \partial \Omega \}
\end{align}

and the operator
\begin{align*}
\mathcal{A}(z,w,y, g_1, (N-2)g_2) =\biggl(&\mathcal{L}(z,w, y) + (0,0, \nabla p)+( y\cdot \nabla z +\nabla \cdot (z\nabla w), zw+ y\cdot \nabla w, (y\cdot \nabla) y) \\
& - ( 0, g_1 \chi_1, (N-2)g_2\chi_2), z(.,0), w(.,0), y(.,0) \biggl),
\end{align*}
$(z,w,y, p, g_1, (N-2)g_2)) \in E$. We have
$$
\mathcal{A}'(0,0,0, 0,0) =\biggl(\mathcal{L}(z,w, y)  + (0,0, \nabla p)  - ( 0, g_1 \chi_1, (N-2)g_2\chi_2), z(.,0), w(.,0), y(.,0) \biggl),
$$
for all $(z,w,y, p, g_1, (N-2)g_2)) \in E$.

In order to apply  Theorem \ref{teoremadecontrol}  to our problem, we must check that the previous framework fits the regularity required. This is done using the following proposition.

\begin{proposition}\label{AC1}
$\mathcal{A} \in C^1(E;G)$.
\end{proposition}

\begin{proof}
All terms appearing in  $\mathcal{A}$ are linear (and consequently $C^1$), except for $( y\cdot \nabla z +\nabla \cdot (z\nabla w), zw+ y\cdot \nabla w, (y\cdot \nabla) y)$. However, the operator
\be\label{bilinear-fe}
\bigl((z,w, y,g_1, g_2), (\tilde{z},\tilde{w}, \tilde{g}_1, \tilde{g}_2) \bigl) \mapsto ( y\cdot \nabla \tilde{z} +\nabla \cdot (z\nabla \tilde{w}), z\tilde{w}+ y\cdot \nabla \tilde{w}, (y\cdot \nabla) \tilde{y})
\ee
is bilinear, so it suffices to prove its continuity from $E\times E$ to $X$.

In fact, we have
\begin{align*}
||\nabla \cdot (z\nabla \tilde{w})||_{X_1} & = ||z\Delta \tilde{w} + \nabla z \cdot \nabla \tilde{w}||_{L^2(e^{-5s\widehat{\beta}}\widehat{\gamma}^{6};Q)}  \\
&\leq C  ||e^{-5/2s\widehat{\beta}} \widehat{\gamma}^{3} z\Delta \tilde{w}||_{L^2(Q)} +  ||e^{-5/2s\widehat{\beta}}\widehat{\gamma}^{3} \nabla z\cdot \nabla \tilde{w}||_{L^2(Q)} \\
&\leq  C \biggl( ||e^{-5/4\widehat{\beta}} \widehat{\gamma}^{13/4}ze^{-5/4\hat{\beta}} \widehat{\gamma}^{-1/4} \Delta \tilde{w}||_{L^2(Q)} \nonumber \\
& \ \ \ +  ||e^{-5/4s\widehat{\beta}}\widehat{\gamma}^{13/4} \nabla ze^{-5/4s\widehat{\beta}}\widehat{\gamma}^{-1/4} \nabla \tilde{w}||_{L^2(Q)} \biggl) \\
& \leq C \biggl( ||e^{-5/4s\widehat{\beta}} \widehat{\gamma}^{13/4}z||_{L^{\infty}(0,T; H^1(\Omega))}||e^{-5/4s\widehat{\beta}} \widehat{\gamma}^{-1/4} \Delta \tilde{w}||_{L^2(0,T; H^1(\Omega))} \\
&  \ \ \  + ||e^{-5/4s\widehat{\beta}}\widehat{\gamma}^{13/4} \nabla z||_{L^{\infty}(0,T; \bold{L}^2(\Omega))}||e^{-5/4s\widehat{\beta}}\widehat{\gamma}^{-1/4} \nabla \tilde{w}||_{L^2(0,T; \bold{H}^2(\Omega))}\biggl),
\end{align*}
for a positive constant $C$.

For the other term, we have
\begin{align*}
|| y\cdot \nabla \tilde{z}||_{X_1} & = || y\cdot \nabla \tilde{z}||_{L^2(e^{-5s\widehat{\beta}}\hat{\gamma}^{6};Q)} \\
& \leq C|| e^{-5/4s\widehat{\beta}}\widehat{\gamma}^{13/4} y ||_{L^{\infty}(0,T; \bold{V}))} ||e^{-5/4s\widehat{\beta}}\widehat{\gamma}^{13/4} \nabla \tilde{z}||_{L^{2}(0,T; \bold{H}^1(\Omega))} \\
& \leq C|| e^{-3/2s\widehat{\beta}}\widehat{\gamma}^{-2-2/11} y ||_{L^{\infty}(0,T; \bold{V}))} ||e^{-5/4s \widehat{\beta}} \widehat{\gamma}^{13/4} \nabla \tilde{z}||_{L^{2}(0,T; \bold{H}^1(\Omega))}
\end{align*}
Analogousy,
\begin{align*}
||z\tilde{w}||_{X_2}+ || y\cdot \nabla \tilde{w})||_{X_2} & = || y\cdot \nabla \tilde{w}||_{L^2(e^{-5s\widehat{\beta}}\widehat{\gamma}^{4};0,T;H^{1}(\Omega))}+  || z\tilde{w}||_{L^2(e^{-5s\widehat{\beta}}\widehat{\gamma}^{4};0,T;H^{1}(\Omega))}  \\
& \leq C|| | e^{-3/2s\widehat{\beta}}\widehat{\gamma}^{-2-2/11} y ||_{_{L^{\infty}(0,T; \bold{V}))}} ||e^{-5/4s\widehat{\beta}}\widehat{\gamma}^{-1/4} \nabla \tilde{w}||_{L^{2}(0,T; \bold{H}^2(\Omega))} \\
& + C||e^{-5/4s\widehat{\beta}}\widehat{\gamma}^{13/4}  z||_{L^{2}(0,T; H^1(\Omega))}  ||e^{-5/4s\widehat{\beta}}\widehat{\gamma}^{-1/4} \nabla \tilde{w}||_{L^{2}(0,T; \bold{H}^2(\Omega))}
\end{align*}
Finally, for the last term, we have
\begin{align*}
|| (y\cdot \nabla) \tilde{y})||_{X_3} & \leq   C||(y\cdot \nabla) \tilde{y}) ||_{L^2(e^{-5s\widehat{\beta}}\widehat{\gamma}^{-5}; Q)} \\
& \leq  C || e^{-3/2s\widehat{\beta}}\widehat{\gamma}^{-2-2/11} y ||_{L^{\infty}(0,T; \bold{V}))} || e^{-3/2s\widehat{\beta}}\widehat{\gamma}^{-2-2/11} \nabla \tilde{y} ||_{L^{2}(0,T; \bold{L}^2(\Omega)))}
\end{align*}

Therefore, continuity of \eqref{bilinear-fe} is established and the  proof Proposition \ref{AC1} is finished.
\end{proof}

An application of  Theorem \ref{teoremadecontrol}  gives the existence of $\delta, \eta > 0$  such that  if $||(n_0 - M,c_0 - M_0, u_0)|| \leq \eta/(K_0^{-1}-\delta)$, then there exists a control $(g_1, (N-2)g_2)$ such that the associated solution $(z,w, y, p)$ to \eqref{IFS} verifies $z(T) = w(T) = 0, y(T)=0$ and $||(z, w, y, g_1, (N-2)g_2)||_E \leq \eta$.  This concludes the proof of  Theorem \ref{mainresultt}.


\appendix

\section{Some technical results}

In this section, we state some technical  results we used along this paper.

The first result will be a Carleman estimate for the solutions of the parabolic equation:
\begin{equation}\label{eq:heatnonhom}
 u_t - \Delta u = f_0 + \sum_{j=1}^N \partial_j f_j \ \mbox{in} \ Q,
\end{equation}
where $f_0,f_1,\dots,f_N\in L^2(Q)$.

The following result is proved in \cite[Theorem 2.1]{IMYamPuel}.
\begin{lemma}\label{teo:Cnonhom}
There exists a constant $\widehat\lambda_0$ only depending on $\Omega$, $\omega_0$, $\eta_0$ and $\ell$ such that for any $\lambda>\widehat\lambda_0$ there exist two constants $C(\lambda)>0$ and $\widehat{s}(\lambda)$, such that for every $s\geq \widehat{s}$ and every $u\in L^2(0,T;H^1(\Omega))\cap H^1(0,T;H^{-1}(\Omega))$ satisfying \eqref{eq:heatnonhom}, we have
\begin{multline}\label{eq:Cnonhom}
s^{-1}\int \! \! \! \int_Q e^{2s\alpha} \phi^{-1}|\nabla u|^2 dx dt + s\int \! \! \! \int_Q e^{2s\alpha} \phi |u|^2 dx dt \\
\leq C\left( s^{-1/2} \|e^{s\alpha}\phi^{-1/4}u\|^2_{H^{1/4,1/2}(\Sigma)} + s^{-1/2} \|e^{s\alpha}\phi^{-1/4 + 1/11}u\|^2_{L^2(\Sigma)} \right.\\
 + s^{-2}\int \! \! \! \int_Q e^{2s\alpha}\phi^{-2}|f_0|^2dx dt
 +\sum_{j=1}^N\int \! \! \! \int_Q e^{2s\alpha}|f_j|^2 dx dt \\
\left.+ s\int \! \! \! \int_{\omega_0\times (0,T)}e^{2s\alpha}\phi|u|^2 dx dt \right).
\end{multline}
\end{lemma}

Recall that
$$\|u\|_{H^{\frac{1}{4},\frac{1}{2}}(\Sigma)}=\left(\|u\|^2_{H^{1/4}(0,T;L^2(\partial\Omega))} + \|u\|^2_{L^{2}(0,T;H^{1/2}(\partial\Omega))} \right)^{1/2}.$$

We now state a Carleman estimate for solutions of the heat equation with homogeneous Neumann boundary condition.
\begin{lemma}\label{lemma-4-CSG}
There exist $C = C(\Omega, \omega')$ and  $\lambda_0 = \lambda_0(\Omega, \omega')$ such that,  for every $\lambda \geq \lambda_0$, there exists $s_0 = s_0(\Omega, \omega', \lambda)$ such that, for any $s \geq s_0(T^{11} + T^{22})$, any $q_0 \in L^2(\Omega)$ and any  $f\in L^2(\Omega)$, the weak solution to
\begin{equation}\label{heat-neumann}
\left |
\begin{array}{ll}
 q_{t}  - \Delta q  = f  &     \mbox{in}  \  \ Q,  \\
\frac{\partial q}{\partial \nu} = 0     &    \mbox{on}  \  \   \Sigma, \\
q(x,0) = q_0&    \mbox{in}    \  \  \Omega,
\end{array}
\right.
\end{equation}
satisfies
\begin{align}
I_{\beta}(s, q) & \leq C\biggl(s^\beta \iint\limits_Q e^{2s\alpha}\phi^{\beta}|f|^2 dxdt  \nonumber + s^{\beta+3} \iint\limits_{\omega' \times (0,T)} e^{2s\alpha}{\phi}^{\beta+3}|q|^2 dxdt \biggl),
\end{align}
for all $\beta \in \mathbb{R}$.
\end{lemma}

The proof of Lemma  \ref{lemma-4-CSG} can be deduced from the Carleman inequality for the heat equation with homogeneous Neumann boundary conditions given in \cite{F-Im}.

\null

The next technical result is a particular case of \cite[Lemma 3]{CorGue}.

\begin{lemma}\label{lemmaCoron-Gue}
Let $\beta \in \mathbb{R}$. There exists $C = C(\lambda)>0$ depending only on $\Omega$, $\omega_0$, $\eta_0$ and $\ell$ such that, for every $\lambda \geq 1$, there exist $\widehat{s}_1(\lambda)$ such, for any $ s \geq \widehat{s}_1(\lambda)$, every $T>0$ and every $u\in L^2(0,T;H^1(\Omega))$, we have
\begin{align}\label{eq:lemma1}
s^{3+\beta}\iint\limits_Q &e^{2s\alpha}\phi^{3+\beta}|u|^2 dx dt \nonumber \\
&\leq C \left( s^{1+\beta} \iint\limits_Q e^{2s\alpha}\phi^{1+\beta}|\nabla u|^2 dx dt + s^{3+\beta} \int\limits_0^T\int\limits_{\omega_0} e^{2s\alpha}\phi^{3+\beta}|u|^2dx dt \right).
\end{align}
\end{lemma}

\begin{remark}
In \cite{CorGue}, slightly different weight functions are used to prove Lemma \ref{lemmaCoron-Gue} Indeed, the authors take $\ell(t)=t(T-t)$. However, this does not change the result since for proving this result we only use integration by parts in the space variable.
\end{remark}

We now present two regularity results for the Stokes system (see \cite{Ladbook}).
\begin{lemma}\label{regStokes}
For every $T>0$ and every $F\in \bold{L}^2(Q)$, there exists a unique solution $u\in L^2(0,T;\bold{H}^2(\Omega))\cap H^1(0,T; \bold{H})$ to the Stokes system
\begin{equation}\label{systemA}
\left |
\begin{array}{ll}
u_t - \Delta u + \nabla p = F  &     \mbox{in}  \  \ Q,  \\
\nabla \cdot u =0 &     \mbox{in}  \  \ Q, \\
 u = 0    &    \mbox{on}  \  \   \Sigma, \\
u(x,0) = 0  &    \mbox{in}    \  \  \Omega,
\end{array}
\right.
\end{equation}
for some $p \in L^2(0,T;H^1(\Omega))$ and there exists a constant $C>0$, depending only on $\Omega$, such that
$$
||u||_{L^2(0,T;\bold{H}^2(\Omega))} + ||u||_{H^1(0,T; \bold{H})} \leq C||F||_{\bold{L}^2(Q)}.
$$
Moreover, if $F \in L^2(0,T;\bold{H}^2(\Omega))\cap H^1(0,T; \bold{L}^2(\Omega))$ and satisfies the compatibility condition
$$
\nabla p_F = F(0) \ \text{in} \ \partial\Omega,
$$
where $p_F$ is any solution of the Neumann boundary value problem
\begin{equation}\label{systemA}
\left |
\begin{array}{ll}
- \Delta p_F = \nabla \cdot F(0)  &     \mbox{in}  \  \ \Omega,  \\
 \frac{\partial p_F}{\partial \nu} = F(0)\cdot \nu     &    \mbox{on}  \  \   \partial \Omega,
 \end{array}
\right.
\end{equation}
then $u \in L^2(0,T;\bold{H}^4(\Omega)) \cap H^1(0,T; \bold{H}^2(\Omega)))$
and there exists a constant $C>0$, depending only on $\Omega$, such that
$$
||u||_{L^2(0,T;\bold{H}^4(\Omega))} + ||u||_{H^1(0,T; \bold{H}^2(\Omega))} \leq C\biggl(||F||_{L^2(0,T;\bold{H}^2(\Omega))} + ||F||_{H^1(0,T; \bold{L}^2(\Omega))}\biggl).
$$

\end{lemma}

\begin{lemma}\label{regStokes10}
If $F \in L^2(0,T;\bold{V})$, then $u \in L^2(0,T;\bold{H}^3(\Omega)) \cap H^1(0,T; \bold{V}))$
and there exists a constant $C>0$, depending only on $\Omega$, such that
$$
||u||_{L^2(0,T;\bold{H}^3(\Omega))} + ||u||_{H^1(0,T; \bold{V})} \leq C||F||_{L^2(0,T;\bold{V})}.
$$
Furthermore, if $F \in L^2(0,T;\bold{H}^3(\Omega))\cap H^1(0,T; \bold{V})$ 
then  $u \in L^2(0,T;\bold{H}^5(\Omega)) \cap H^1(0,T; \bold{H}^3(\Omega)) \cap H^2(0,T;\bold{V})$ and there exists a constant $C>0$, depending only on $\Omega$, such that
$$
||u||_{L^2(0,T;\bold{H}^5(\Omega))} + ||u||_{H^1(0,T; \bold{H}^3(\Omega))} + ||u||_{H^2(0,T; \bold{V})}  \leq C\bigl(||F||_{L^2(0,T;\bold{H}^3(\Omega))} + ||F_t||_{L^2(0,T; \bold{V})} \bigl).
$$

\end{lemma}

\section{Carleman Inequality for the Stokes operator}\label{proofcarlemanstokes}

In this section we prove Lemma \ref{CarlemanStokes} used in the proof of Theorem \ref{CarlemanP}.

\begin{proof} For better comprehension, we divide the proof into several steps.


\null

\textit{Step 1.} \textit{Estimate of $\nabla \nabla \Delta (z_i), \ i=1,3.$}

\null

We begin noticing that  since $F_3 \in L^2(0,T,\bold{V})$, we have that $\rho' v \in L^2(0,T;\bold{H}^3(\Omega)) \cap H^1(0,T;\bold{V})$ (see Lemma \ref{regStokes10} above). Therefore, we can apply the operator $\nabla \nabla \Delta \cdot$ to the equation of $z_i$ (see \eqref{eq:u2}), $i=1,3$,  to get
\begin{equation}\label{Z1}
-\widehat{Z}_{i,t} - \Delta \widehat{Z}_i= -\nabla \nabla(\Delta(\rho' v_i)) \ \mbox{in} \ Q,
\end{equation}
where $\widehat{Z}_i = \nabla \nabla \Delta z_i$. Here, we have used the fact that $\Delta r=0 \ \mbox{in} \ Q$.

Next, we apply Lemma \ref{teo:Cnonhom} to \eqref{Z1}, with $i=1,3$, and add these estimates. This gives
\begin{align}\label{Z2}
\sum_{i=1,3} \widehat{I}_{-2}(s; \widehat{Z}_i) \leq  \ C&\sum_{i=1,3} \biggl( s^{-\frac{1}{2}} \|e^{s\alpha}\phi^{-\frac{1}{4}}\widehat{Z}_i\|^2_{\bold{H}^{\frac{1}{4},\frac{1}{2}}(\Sigma)} + s^{-\frac{1}{2}} \|e^{s\alpha}\phi^{-1/4 + 1/11}\widehat{Z}_i\|^2_{\bold{L}^2(\Sigma)}  \\
& + \int \! \! \! \int_Q e^{2s\alpha} |\rho'|^2|\nabla \Delta v_i|^2 dx dt + s\int \! \! \! \int_{\omega^1_0 \times (0,T)}e^{2s\alpha}\phi|\widehat{Z}_i|^2 dx dt \biggr). \nonumber
\end{align}
 Notice that this can be done because the right-hand side of \eqref{Z1} belongs to $L^2(0,T;\bold{H}^{-1}(\Omega))$.

Now, using  Lemma \ref{lemmaCoron-Gue}, with $\beta =0$, we see that
\begin{align}\label{A1-1}
\sum_{i=1,3} s^3\int \! \! \! \int_Q& e^{2s\alpha}\phi^3| Z_i|^2dx dt \nonumber  \\
&\leq C\sum_{i=1,3}  \left( s \int \! \! \! \int_Qe^{2s\alpha}\phi|\widehat{Z}_i|^2dx dt + s^3\int \! \! \! \int_{\omega^2_0 \times (0,T) } e^{2s\alpha}\phi^3|Z_i|^2 dxdt \right),
\end{align}
for every $s\geq C_1$, where  $Z_i := \nabla \Delta z_i$.

In \eqref{Z2}, we estimate the  local integral of $\widehat{Z}_i$, $i=1,3$, as follows:
\begin{align}
& s\int \! \! \! \int_{\omega_0^1 \times (0,T)}e^{2s\alpha}\phi|\widehat{Z}_i|^2 dx dt \leq  s\int \! \! \! \int_{\omega_0^2 \times (0,T)} \theta_2 e^{2s\alpha}\phi|\nabla Z_i|^2 dx dt \nonumber \\
&= \frac{s}{2}\int \! \! \! \int_{\omega_0^2 \times (0,T)}\Delta( e^{2s\alpha}\phi \theta_2 ) |Z_i|^2  dx dt  - s\int \! \! \! \int_{\omega_0^2 \times (0,T)}\theta_2e^{2s\alpha}\phi \nabla \cdot(\widehat{Z}_i) Z_idx dt  \nonumber \\
& \leq  Cs^3\int \! \! \! \int_{\omega_0^2 \times (0,T) } e^{2s\alpha}\phi^3|Z_i|^2 dxdt +\delta  \widehat{I}_{-2}(s; \widehat{Z}_i) ,
\end{align}
for any $\delta >0$, since
$$
|\Delta( e^{2s\alpha}\phi \theta_2 )| \leq Cs^2\phi^3 e^{2s\alpha}1_{\omega_0^2}.
$$

Hence,
\begin{align}\label{Z3}
\sum_{i=1,3}    \biggl( s^3   \int \! \! \! \int_Q   e^{2s\alpha}&\phi^3| Z_i|^2dx dt +  \widehat{I}_{2}(s; \widehat{Z}_i)\biggr) \nonumber \\
\leq \  C  \ \sum_{i=1,3} & \biggl( s^{-\frac{1}{2}} \|e^{s\alpha}\phi^{-\frac{1}{4}}\widehat{Z}_i\|^2_{\bold{H}^{\frac{1}{4},\frac{1}{2}}(\Sigma)} + s^{-\frac{1}{2}} \|e^{s\alpha}\phi^{-1/4+1/11}\widehat{Z}_i\|^2_{\bold{L}^2(\Sigma)}   \\
& + \int \! \! \! \int_Q e^{2s\alpha} |\rho'|^2|\nabla \Delta v_i|^2 dx dt + s^3\int \! \! \! \int_{\omega_0^2 \times (0,T) } e^{2s\alpha}\phi^3|Z_i|^2 dxdt \biggr). \nonumber
\end{align}

\null

Using again Lemma \ref{lemmaCoron-Gue}, with $\beta =2$, $i=1,3$, we get
\begin{align}\label{Z4}
s^5\int \! \! \! \int_Q e^{2s\alpha}\phi^5|\Delta z_i|^2dx dt   \leq C &  \biggl( s^3 \int \! \! \! \int_Qe^{2s\alpha}\phi^3|Z_i|^2dx dt   \nonumber  \\
&+ s^5\int \! \! \! \int_{\omega_0^3 \times (0,T) } e^{2s\alpha}\phi^5|\Delta z_i|^2 dxdt \biggr),
\end{align}
for every $s\geq C_1$.

From \eqref{Z3} and \eqref{Z4}, we obtain
\begin{align}\label{Z5}
\sum_{i=1,3}  \biggl( s^5\int \! \! \! \int_Q & e^{2s\alpha}\phi^5|\Delta z_i|^2dxdt + s^3\int \! \! \! \int_Q e^{2s\alpha}\phi^3| Z_3|^2dx dt+  \widehat{I}_{-2}(s; \widehat{Z}_3) \biggr) \nonumber \\
 \leq C \sum_{i=1,3} &\biggl( s^{-\frac{1}{2}} \|e^{s\alpha}\phi^{-\frac{1}{4}}\widehat{Z}_i\|^2_{\bold{H}^{\frac{1}{4},\frac{1}{2}}(\Sigma)} + s^{-\frac{1}{2}} \|e^{s\alpha}\phi^{-1/4+1/11}\widehat{Z}_i\|^2_{\bold{L}^2(\Sigma)}  \\
 + \int \! \! \! \int_Q& e^{2s\alpha} |\rho'|^2|\nabla \Delta v_i|^2 dx dt + s^3\int \! \! \! \int_{\omega_0^2 \times (0,T) } e^{2s\alpha}\phi^3|Z_i|^2 dxdt  \nonumber \\
&+ s^5\int \! \! \! \int_{\omega_0^3 \times (0,T) } e^{2s\alpha}\phi^5|\Delta z_i|^2 dxdt\biggr), \nonumber
\end{align}
for every $s\geq C_1$.

\null


\textit{Step 2.} {\textit{Estimate of $\nabla \Delta v_i$, $i=1,3$.}}

\null

By \eqref{eq:regularity} and the fact that $s^{11/5} e^{2s\widehat{\alpha}} \widehat{\phi}^{11/5}$ is bounded, we  estimate the integrals involving  $	\nabla \Delta v_i$, $i=1,3$, on the right-hand side of \eqref{Z5}. Indeed,

\begin{align}\label{Z6}
\int \! \! \! \int_Q e^{2s\alpha} |\rho'|^2|\nabla \Delta v_i|^2 dx dt &= \int \! \! \! \int_Qe^{2s\alpha} |\rho'|^2 |\rho|^{-2} |\nabla\Delta (\rho v_i)|^2 dx dt \nonumber \\
&\leq C \biggl( s^{2+2/11}\int \! \! \! \int_Q e^{2s\widehat{\alpha}}\widehat{\phi}^{2+2/11}|\nabla \Delta w_i|^2 dx dt + s^{2+2/11}\int \! \! \! \int_Q e^{2s\widehat{\alpha}} \widehat{\phi}^{2+2/11}|Z_i|^2 dx dt\biggr) \nonumber \\
&\leq C\biggl(\|\rho F_3\|^2_{L^2(0,T;\bold{V})} + s^{2+2/11}\int \! \! \! \int_Q e^{2s\widehat{\alpha}} \widehat{\phi}^{3}|Z_i|^2 dx dt \biggr).
\end{align}
since
$$
|\alpha_t|^2 \leq CT^2\phi^{2+2/11}.
$$

Therefore, from \eqref{Z5} and \eqref{Z6}, we  have
\begin{align}\label{Z7}
\sum_{i=1,3}\biggl(s^5\int \! \! \! \int_Q & e^{2s\alpha}\phi^5|\Delta z_i|^2dxdt + s^3\int \! \! \! \int_Q e^{2s\alpha}\phi^3| Z_i|^2dx dt+  \widehat{I}_{-2}(s; \widehat{Z}_i) \biggr) \nonumber \\
 \leq& \ C\sum_{i=1,3}\biggl( \|\rho F_3\|^2_{L^2(0,T;\bold{V})}+ s^{-\frac{1}{2}} \|e^{s\alpha}\phi^{-\frac{1}{4}}\hat{Z}_i\|^2_{\bold{H}^{\frac{1}{4},\frac{1}{2}}(\Sigma)} + s^{-\frac{1}{2}} \|e^{s\alpha}\phi^{-1/4+1/11}\widehat{Z}_i\|^2_{\bold{L}^2(\Sigma)} \nonumber \\
 &+ s^3\int \! \! \! \int_{\omega_0^2 \times (0,T) } e^{2s\alpha}\phi^3|Z_i|^2 dxdt  + s^5\int \! \! \! \int_{\omega_0^3 \times (0,T) } e^{2s\alpha}\phi^5|\Delta z_i|^2 dxdt
 \biggr).
\end{align}


\textit{Step 3.} \textit{Estimate of a global term of $z_2$}.

\null

From the fact that $\Delta$ defines a norm in $H^2(\Omega)\times H^1_0(\Omega)$, we have
\begin{align}
s^5\int \! \! \! \int_Qe^{2s\widehat{\alpha}}\widehat{\phi}^5 | z_2|^2 dx dt \leq C\sum_{i=1,3} s^5\int \! \! \! \int_Q e^{2s\widehat{\alpha}}\widehat{\phi}^5 |\Delta z_i|^2dx dt,
\end{align}
since $z|_{\partial\Omega}=0$ and $\nabla \cdot z=0$.

Hence,
\begin{align}\label{Z8}
s^5\int \! \! \! \int_Qe^{2s\widehat{\alpha}}&\widehat{\phi}^5 | z_2|^2 dx dt  + \sum_{i=1,3}\biggl(s^5\int \! \! \! \int_Q  e^{2s\alpha}\phi^5|\Delta z_i|^2dxdt + s^3\int \! \! \! \int_Q e^{2s\alpha}\phi^3| Z_i|^2dx dt+  \widehat{I}_{-2}(s; \widehat{Z}_i) \biggr) \nonumber \\
 \leq& \ C\sum_{i=1,3}\biggl( \|\rho F_3\|^2_{L^2(0,T;\bold{V})}+ s^{-\frac{1}{2}} \|e^{s\alpha}\phi^{-\frac{1}{4}}\widehat{Z}_i\|^2_{\bold{H}^{\frac{1}{4},\frac{1}{2}}(\Sigma)} + s^{-\frac{1}{2}} \|e^{s\alpha}\phi^{-1/4+1/11}\widehat{Z}_i\|^2_{\bold{L}^2(\Sigma)} \nonumber \\
 &+ s^3\int \! \! \! \int_{\omega_0^2 \times (0,T) } e^{2s\alpha}\phi^3|Z_i|^2 dxdt  + s^5\int \! \! \! \int_{\omega_0^3 \times (0,T) } e^{2s\alpha}\phi^5|\Delta z_i|^2 dxdt
 \biggr).
\end{align}


\textit{Step 4.} \textit{Estimate of the local integral of $Z_i$, $i=1,3$.}

\null

We have
\begin{align}
& s^3\int \! \! \! \int_{\omega_0^2 \times (0,T)}e^{2s\alpha}\phi^3|Z_i|^2 dx dt \leq  s^3\int \! \! \! \int_{\omega_0^3 \times (0,T)} \theta_3 e^{2s\alpha}\phi^3|\nabla 	\Delta z_i|^2 dx dt \nonumber \\
&= \frac{s^3}{2}\int \! \! \! \int_{\omega_0^3 \times (0,T)}\Delta( e^{2s\alpha}\phi^3 \theta_3 )  |\Delta z_i|^2 dx dt  - s^3\int \! \! \! \int_{\omega_0^3 \times (0,T)}\theta_3e^{2s\alpha}\phi^3 \nabla \cdot Z_i \Delta z_3dx dt  \nonumber \\
& \leq Cs^5\int \! \! \! \int_{\omega_0^3 \times (0,T) } e^{2s\alpha}\phi^5|\Delta z_i|^2 dxdt +  \delta  \widehat{I}_{-2}(s; \widehat{Z}_i),
\end{align}
since
$$
|\Delta( e^{2s\alpha}\phi^3 \theta_3 )| \leq Cs^2\phi^5e^{2s\alpha}1_{\omega_0^3}.
$$

From \eqref{Z8}, we get
\begin{align}\label{Z9}
s^5\int \! \! \! \int_Qe^{2s\widehat{\alpha}}\widehat{\phi}^5 | z_2|^2 dx dt  +& \sum_{i=1,3}\biggl(s^5\int \! \! \! \int_Q  e^{2s\alpha}\phi^5|\Delta z_i|^2dxdt + s^3\int \! \! \! \int_Q e^{2s\alpha}\phi^3| Z_i|^2dx dt+  \widehat{I}_{-2}(s; \widehat{Z}_i) \biggr) \nonumber \\
 \leq \ C&\sum_{i=1,3}\biggl( \|\rho F_3\|^2_{L^2(0,T;\bold{V})}+ s^{-\frac{1}{2}} \|e^{s\alpha}\phi^{-\frac{1}{4}}\widehat{Z}_i\|^2_{\bold{H}^{\frac{1}{4},\frac{1}{2}}(\Sigma)} + s^{-\frac{1}{2}} \|e^{s\alpha}\phi^{-1/4+1/11}\widehat{Z}_i\|^2_{\bold{L}^2(\Sigma)} \nonumber \\
 & + s^5\int \! \! \! \int_{\omega_0^3 \times (0,T) } e^{2s\alpha}\phi^5|\Delta z_i|^2 dxdt
 \biggr).
\end{align}


\textit{Step 5} \textit{Estimate of the $\bold{L}^2$ boundary terms.}

Using the fact that
\begin{align}
||e^{s\widehat{\alpha}}\widehat{Z}_i ||^2_{\bold{L}^2(\Sigma)}& \leq ||e^{s\widehat{\alpha}}\widehat{Z}_i ||^2_{\bold{L}^2(Q)} +  ||s^{1/2}e^{s\widehat{\alpha}}\widehat{\phi}^{1/2}\widehat{Z}_i||_{\bold{L}^2(Q)} ||s^{-1/2}e^{s\widehat{\alpha}}\widehat{\phi}^{-1/2}\nabla \widehat{Z}_i||_{\bold{L}^2(Q)} 
\end{align}
it is not difficult to see that we can absorb $s^{-\frac{1}{2}}\|e^{s\alpha}\phi^{-1/4+1/11}\widehat{Z}_i\|^2_{\bold{L}^2(\Sigma)}$ in \eqref{Z9} by taking $s$ large enough.

\null

\textit{Step 6.} \textit{Estimate of the $\bold{H}^{\frac{1}{4},\frac{1}{2}}$ boundary terms.}

\null
 To eliminate the $\bold{H}^{\frac{1}{4},\frac{1}{2}}$ boundary terms, we show that $z_i$, $i=1,3$, multiplied by several weight functions are regular enough. We begin noticing   that, from \eqref{eq:regularity},  we have
\begin{align}\label{B1}
 s^5 \int \! \! \! \int_Q & e^{2s\widehat{\alpha}}\widehat{\phi}^5|\rho|^2|v|^2 dx dt   \leq C \bigl( \|\rho F_3\|^2_{L^2(0,T;\bold{V})} + s^5 \int \! \! \! \int_Q e^{2s\widehat{\alpha}}\widehat{\phi}^5|z|^2 dx dt \bigl).
\end{align}
Thus, the term $\|s^{5/2}e^{s\widehat{\alpha}}\widehat{\phi}^{5/2}\rho v\|^2_{\bold{L}^2(Q)}$ is bounded by the left-hand side of \eqref{Z10} and $\|\rho F_3 \|^2_{L^2(0,T;\bold{V})}$:
\begin{align}\label{Z10}
s^5\int \! \! \! \int_Q&e^{2s\widehat{\alpha}}\widehat{\phi}^5 | z_2|^2 dx dt + s^5 \int \! \! \! \int_Q e^{2s\widehat{\alpha}}\widehat{\phi}^5|\rho|^2|v|^2 dx dt\nonumber \\
+& \sum_{i=1,3}\biggl(s^5\int \! \! \! \int_Q  e^{2s\alpha}\phi^5|\Delta z_i|^2dxdt + s^3\int \! \! \! \int_Q e^{2s\alpha}\phi^3| Z_i|^2dx dt+  \widehat{I}_{-2}(s; \widehat{Z}_i) \biggr)  \\
 \leq \ C&\sum_{i=1,3}\biggl( \|\rho F_3\|^2_{L^2(0,T;\bold{V})}+ s^{-\frac{1}{2}} \|e^{s\alpha}\phi^{-\frac{1}{4}}\widehat{Z}_i\|^2_{\bold{H}^{\frac{1}{4},\frac{1}{2}}(\Sigma)}  + s^5\int \! \! \! \int_{\omega_0^3 \times (0,T) } e^{2s\alpha}\phi^5|\Delta z_i|^2 dxdt. \nonumber
 \biggr).
\end{align}

We define now
$$\widetilde{z}:=\widetilde{l}(t)z,\,\widetilde{r}:=\widetilde{l}(t)r,$$
with
$$
\widetilde{l}(t) =  s^{3/2-1/11 }\widehat{\phi}^{3/2-1/11}e^{s\widehat{\alpha}}. 
$$
From \eqref{eq:u2}, we see that $(\widetilde{z},\widetilde{r})$ is the solution of the Stokes system:
\begin{equation*}
\left\lbrace \begin{array}{ll}
    -\widetilde{z}_t - \Delta \widetilde{z} + \nabla \widetilde{r} = -\widetilde{l}\rho' v - \widetilde{l}' z & \mbox{ in }Q, \\
    \nabla\cdot \widetilde{z} = 0 & \mbox{ in }Q, \\
    \widetilde{z} = 0 & \mbox{ on }\Sigma, \\
    \widetilde{z}(T) = 0 & \mbox{ in }\Omega.
 \end{array}\right.
\end{equation*}

Taking into account that
$$|\widehat{\alpha}_t| \leq C T\widehat{\phi}^{1+1/11},\,|\rho'|\leq C s^{1+1/11}\widehat{\phi}^{1+1/11}\rho,$$
$$
|\widetilde{l}\rho'| \leq C s^{5/2}\widehat{\phi}^{5/2}e^{s\widehat{\alpha}}\rho, \ \ |\widetilde{l}'| \leq Cs^{5/2}\widehat{\phi}^{5/2}e^{s\widehat{\alpha}},
$$
and using Lemma \ref{regStokes}, we have that
$$
\widetilde{z} \in L^2(0,T;\bold{H}^2(\Omega))\cap H^1(0,T;\bold{L}^2(\Omega))
$$
and
\begin{align}\label{tildez}
\|\widetilde{z}\|^2_{L^2(0,T;\bold{H}^2(Q))\cap H^1(0,T;\bold{L}^2(Q))}  \leq C \left( \| s^{5/2}\widehat{\phi}^{5/2}e^{s\widehat{\alpha}}\rho v \|^2_{\bold{L}^2(Q)} + \|s^{5/2}\widehat{\phi}^{5/2}e^{s\widehat{\alpha}} z\|^2_{\bold{L}^2(Q)} \right),
\end{align}
thus, $\|\widetilde{l}z\|^2_{L^2(0,T;\bold{H}^2(Q))\cap H^1(0,T;\bold{L}^2(Q))}$ is bounded by the left-hand side of \eqref{Z10}.


 Next, let
$$z^*:=l^*(t)z,\, r^*:=l^*(t)r,$$
with
$$
l^*(t) =  s^{1/2-2/11}\widehat{\phi}^{1/2-2/11}e^{s\widehat{\alpha}}.
$$
From \eqref{eq:u2}, $(z^*,r^*)$ is the solution of the Stokes system:
\begin{equation*}
\left\lbrace \begin{array}{ll}
    -z^*_t - \Delta z^* + \nabla r^* = -l^*\rho' v - (l^*)' z & \mbox{ in }Q, \\
    \nabla\cdot z^* = 0 & \mbox{ in }Q, \\
    z^* = 0 & \mbox{ on }\Sigma, \\
    z^*(T) = 0 & \mbox{ in }\Omega.
 \end{array}\right.
\end{equation*}
Let us show that the right-hand side of this system is in  $L^2(0,T;\bold{H}^2(\Omega))\cap H^1(0,T;\bold{L}^2(\Omega))$.

For the first term, we write
\be\label{defvstar}
l^*\rho' v = l^*\rho' \widetilde{l}^{-1}\rho^{-1}\widetilde{l}\rho  v =l^*\rho' \widetilde{l}^{-1}\rho^{-1} (\widetilde{z}+\widetilde{l}w)
\ee
and since
$$
|l^*\rho' \widetilde{l}^{-1}\rho^{-1}|  \leq C,
$$
we see that $l^*\rho' v =L^2(0,T;\bold{H}^2(\Omega))$. Moreover, because
$$
|(l^*\rho' \widetilde{l}^{-1}\rho^{-1})'|  \leq Cs\widehat{\phi}^{1+1/11}
$$
 the regularity of $\widetilde{z}$ and the one of $w$ give
$$
l^*\rho'v \in H^1(0,T;\bold{L}^2(\Omega)).
$$
From \eqref{defvstar}, \eqref{eq:regularity} and \eqref{tildez}, we have
\begin{align}\label{lvsatarprimeb}
\|l^*\rho'v\|^2_{L^2(0,T;\bold{H}^2(Q))\cap H^1(0,T;\bold{L}^2(Q))}  \leq & \ C \left( \| s^{5/2}\widehat{\phi}^{5/2}e^{s\widehat{\alpha}}\rho v \|^2_{\bold{L}^2(Q)} + \|s^{5/2}\widehat{\phi}^{5/2}e^{s\widehat{\alpha}} z\|^2_{\bold{L}^2(Q)} \right) \nonumber \\
& + C\|\rho F_3\|^2_{L^2(0,T;\bold{V})},
\end{align}


For the other term, we write
\be\label{derivstarz}
(l^*)'z = \widetilde{l}^{-1}(l^*)' \widetilde{z}
\ee
and since
$$
|\widetilde{l}^{-1}(l^*)' |\leq C
$$
we have that $(l^*)'z \in L^2(0,T;\bold{H}^2(\Omega))$. From the regularity of $\widetilde{z}$, and the fact that
$$
|((l^*)' \widetilde{l}^{-1})'|  \leq Cs\widehat{\phi}^{1+1/11},
$$
we have that $ (l^*)' z \in H^1(0,T;\bold{L}^2(\Omega))$ and
\begin{align}\label{tlstarzprime}
\| (l^*)' z\|^2_{L^2(0,T;\bold{H}^2(Q))\cap H^1(0,T;\bold{L}^2(Q))}  \leq & \ C \left( \| s^{5/2}\widehat{\phi}^{5/2}e^{s\widehat{\alpha}}\rho v \|^2_{\bold{L}^2(Q)} + \|s^{5/2}\widehat{\phi}^{5/2}e^{s\widehat{\alpha}} z\|^2_{\bold{L}^2(Q)} \right).
\end{align}


Using Lemma \ref{regStokes} once more, we obtain
$$
z^* \in L^2(0,T; \bold{H}^4(\Omega))\cap H^1(0,T;\bold{H}^2(\Omega))
$$
and
\begin{align}\label{starriest}
\|z^*\|^2_{L^2(0,T;\bold{H}^4(Q))\cap H^1(0,T;\bold{H}^2(Q))}  \leq & \ C \left( \| s^{5/2}\widehat{\phi}^{5/2}e^{s\widehat{\alpha}}\rho v \|^2_{\bold{L}^2(Q)} + \|s^{5/2}\widehat{\phi}^{5/2}e^{s\widehat{\alpha}} z\|^2_{\bold{L}^2(Q)} \right) \nonumber \\
& + C\|\rho F_3\|^2_{L^2(0,T;\bold{V})}.
\end{align}

Let us now define $\widehat{z} = \widehat{l}z$, where
$$
\widehat{l} = s^{-5/22}\widehat{\phi}^{-5/22}e^{s\widehat{\alpha}}.
$$
From \eqref{eq:u2}, $(\widehat{z},\widehat{r})$ is the solution of the Stokes system:
\begin{equation*}
\left\lbrace \begin{array}{ll}
-\widehat{z} _t - \Delta \widehat{z}  + \nabla \widehat{r}  = -\widehat{l} \rho' v - \widehat{l} ' z & \mbox{ in }Q, \\
    \nabla\cdot \widehat{z}  = 0 & \mbox{ in }Q, \\
    \widehat{z}  = 0 & \mbox{ on }\Sigma, \\
    \widehat{z} (T) = 0 & \mbox{ in }\Omega.
 \end{array}\right.
\end{equation*}

Noticing  that
$$
|\widehat{l}'| \leq C s^{1-5/22+1/11} \widehat{\phi}^{-5/22+1+1/11}e^{s\widehat{\alpha}} = C(l^*)^{1/2}(\widetilde{l})^{1/2}
$$
and interpolating $H^2(\Omega)$ and $H^4(\Omega)$, we obtain
\be \label{EST100}
|| \widehat{l}' z||_{L^2(0,T;\bold{H}^3(\Omega))} \leq C || \widetilde{l} z||_{L^2(0,T;\bold{H}^2(\Omega))}^{1/2}|| l^* z||_{L^2(0,T;\bold{H}^4(\Omega))}^{1/2}
\ee
and $\widehat{l}' z \in L^2(0,T;\bold{H}^3(\Omega))$.

Next, interpolating $L^2(\Omega)$ and $H^2(\Omega)$, we obtain
$$
|| \widehat{l}' z_t||_{L^2(0,T;\bold{H}^1(\Omega))} \leq C || \widetilde{l} z_t||_{L^2(0,T;\bold{L}^2(\Omega))}^{1/2}|| l^* z_t||_{L^2(0,T;\bold{H}^2(\Omega))}^{1/2}.
$$
We also have the following estimate
$$
|| (\widehat{l}')' z||_{L^2(0,T;\bold{H}^1(\Omega))} \leq C||  s^{5/2}e^{s\widehat{\alpha}}\widehat{\phi}^{5/2}z||_{L^2(0,T;\bold{L}^2(\Omega))}^{1/2}|| \widetilde{l} z||_{L^2(0,T;\bold{H}^2(\Omega))}^{1/2}.
$$
Therefore,  the following estimate holds
\begin{align}
|| \widehat{l}' z||_{ L^2(0,T;\bold{H}^3(\Omega)) \cap H^1(0,T;\bold{H}^1(\Omega))}^2 \leq & \ C \left( \| s^{5/2}\widehat{\phi}^{5/2}e^{s\widehat{\alpha}}\rho v \|^2_{\bold{L}^2(Q)} + \|s^{5/2}\widehat{\phi}^{5/2}e^{s\widehat{\alpha}} z\|^2_{\bold{L}^2(Q)}\right.  \nonumber \\
& +\left. \|\rho F_3\|^2_{L^2(0,T;\bold{V})}\right).
\end{align}

Next, writing
$$
\widehat{l} \rho' v = \widehat{l} \rho' (l^*)^{-1/2}\tilde{l}^{-1/2}\rho^{-1}(\tilde{l}^{1/2}(l^*)^{1/2}z +\tilde{l}^{1/2}(l^*)^{1/2}w)
$$
and using the fact that
$$
 |\widehat{l} \rho' (l^*)^{-1/2}\tilde{l}^{-1/2}\rho^{-1}| \leq C,
$$
the regularity of $\tilde{z}$, $z^*$ and the one of $w$, we have that $\widehat{l} \rho'v \in L^2(0,T;\bold{H}^3(\Omega))$ and the following estimate holds
\begin{align}\label{lvhatprimeb}
\|\widehat{l} \rho'v\|^2_{L^2(0,T;\bold{H}^3(\Omega))}  \leq & \ C \left( \| s^{5/2}\widehat{\phi}^{5/2}e^{s\widehat{\alpha}}\rho v \|^2_{\bold{L}^2(Q)} + \|s^{5/2}\widehat{\phi}^{5/2}e^{s\widehat{\alpha}} z\|^2_{\bold{L}^2(Q)} \right. \nonumber \\
& + \left. \|\rho F_3\|^2_{L^2(0,T;\bold{V})}\right).
\end{align}

It is immediate to see that
\begin{align*}
|| ((l^*)^{1/2}\tilde{l}^{1/2} z)_t||_{L^2(0,T;\bold{H}^1(\Omega))} &\leq C || (l^*)^{1/2}\tilde{l}^{1/2} z||_{H^1(0,T;\bold{H}^1(\Omega))} \\
&\leq C|| \widetilde{l} z||_{H^1(0,T;\bold{L}^2(\Omega))}^{1/2}|| l^* z||_{H^1(0,T;\bold{H}^2(\Omega))}^{1/2}
\end{align*}
and because
$$
|(\widehat{l} \rho' (l^*)^{-1/2}\tilde{l}^{-1/2}\rho^{-1})'|  \leq Cs\widehat{\phi}^{1+1/11},
$$
we also have that
$$
|| (\widehat{l} \rho' (l^*)^{-1/2}\tilde{l}^{-1/2}\rho^{-1})'(l^*)^{1/2}\tilde{l}^{1/2} z||_{L^2(0,T;\bold{H}^1(\Omega))} \leq C||  s^{5/2}e^{s\widehat{\alpha}}\widehat{\phi}^{5/2}z||_{L^2(0,T;\bold{L}^2(\Omega))}^{1/2}|| \widetilde{l} z||_{L^2(0,T;\bold{H}^2(\Omega))}^{1/2}.
$$
Hence, $\widehat{l} \rho'v \in H^1(0,T;\bold{H}^1(\Omega))$ and we have the estimate
\begin{align}\label{lvhatfinal0}
\|\widehat{l} \rho'v\|^2_{L^2(0,T;\bold{H}^3(\Omega))\cap H^1(0,T;\bold{H}^1(\Omega))}  \leq & \ C \left( \| s^{5/2}\widehat{\phi}^{5/2}e^{s\widehat{\alpha}}\rho v \|^2_{\bold{L}^2(Q)} + \|s^{5/2}\widehat{\phi}^{5/2}e^{s\widehat{\alpha}} z\|^2_{\bold{L}^2(Q)} \right. \nonumber \\
& +\left. \|\rho F_3\|^2_{L^2(0,T;\bold{V})}\right).
\end{align}

Therefore, from Lemma \ref{regStokes10}, we conclude that
\be \label{REG100}
\widehat{l}z \in  L^2(0,T;\bold{H}^5(\Omega)) \cap H^1(0,T;\bold{H}^3(\Omega))
\ee
and has the estimate
\begin{align}\label{lahatfinal}
\|\widehat{l}z\|^2_{L^2(0,T;\bold{H}^5(\Omega))\cap H^1(0,T;\bold{H}^3(\Omega))}  \leq & \ C \left( \| s^{5/2}\widehat{\phi}^{5/2}e^{s\widehat{\alpha}}\rho v \|^2_{\bold{L}^2(Q)} + \|s^{5/2}\widehat{\phi}^{5/2}e^{s\widehat{\alpha}} z\|^2_{\bold{L}^2(Q)} \right. \nonumber \\
& + \left. \|\rho F_3\|^2_{L^2(0,T;\bold{V})}\right).
\end{align}

Now, since
$$
||\widehat{l} \widehat{Z}_i||_{L^2(0,T;\bold{H}^1(\Omega)) \cap H^1(0,T;\bold{H}^{-1}(\Omega))} \leq C||\widehat{l}z||_{L^2(0,T;\bold{H}^5(\Omega)) \cap H^1(0,T;\bold{H}^3(\Omega))}
$$


and
$$
 s^{-\frac{1}{2}} \|e^{s\alpha}\phi^{-\frac{1}{4}}\widehat{Z}_i\|^2_{\bold{H}^{\frac{1}{4},\frac{1}{2}}(\Sigma)}  \leq s^{-1/22}  \|\widehat{l}\widehat{Z}_i\|^2_{\bold{H}^{\frac{1}{4},\frac{1}{2}}(\Sigma)},
 $$
 the $H^{\frac{1}{4},\frac{1}{2}}$ boundary terms on the right-hand side of \eqref{Z10} can be absorbed by its left-hand side by taking $s$ large enough.


%


Therefore, we conclude that
\begin{align}\label{Z110}
s^5\int \! \! \! \int_Q&e^{2s\widehat{\alpha}}\widehat{\phi}^5 | z_2|^2 dx dt + s^5 \int \! \! \! \int_Q e^{2s\widehat{\alpha}}\widehat{\phi}^5|\rho|^2|v|^2 dx dt\nonumber \\
+& \sum_{i=1,3}\biggl(s^5\int \! \! \! \int_Q  e^{2s\alpha}\phi^5|\Delta z_i|^2dxdt + s^3\int \! \! \! \int_Q e^{2s\alpha}\phi^3| Z_i|^2dx dt+  \widehat{I}_{-2}(s; \widehat{Z}_i) \biggr)  \\
 \leq \ C&\sum_{i=1,3}\biggl( \|\rho F_3\|^2_{L^2(0,T;\bold{V})}  + s^5\int \! \! \! \int_{\omega_0^3 \times (0,T) } e^{2s\alpha}\phi^5|\Delta z_i|^2 dxdt. \nonumber
 \biggr),
\end{align}
which is exactly \eqref{Z11}.
\end{proof}

\section{Proof of Claim \ref{claim1}}\label{proofclaim1}

In this section, we prove Claim \ref{claim1} used in the proof of Theorem \ref{CarlemanP}.

First, we use integration by parts to see that
\begin{align}\label{M1}
  s^9\int \! \! \! \int_{\omega_0^5 \times (0,T)}\theta_5 e^{2s\alpha} \phi^{18}\widehat{\phi}^{-9/2} &\rho\Delta \psi \xi_t dxdt    \nonumber\\
&=  -s^9\int \! \! \! \int_{\omega_0^5 \times (0,T)}\theta_5 e^{2s\alpha}\rho \widehat{\phi}^{-9/2} \phi^{18}(\Delta \psi)_t \xi dxdt  \\
& \ \  \ \ -s^9\int \! \! \! \int_{\omega_0^5\times (0,T)}\theta_5 (e^{2s\alpha}\rho  \phi^{18}\widehat{\phi}^{-9/2})_t  \Delta \psi \xi   dxdt. \nonumber
\end{align}

For the  first term, we use \eqref{x2-1} to write
\begin{align}\label{QX1}
&s^9\int \! \! \! \int_{\omega_0^5 \times (0,T)}\theta_5 e^{2s\alpha}\rho \widehat{\phi}^{-9/2} \phi^{18}(-\Delta \psi)_t \xi dxdt \\
&=s^9\int \! \! \! \int_{\omega_0^5 \times (0,T)}\theta_5 e^{2s\alpha}\rho \widehat{\phi}^{-9/2} \phi^{18} \xi \bigl((\Delta (\Delta \psi)   -M_0e^{-Mt}\rho\widehat{\phi}^{-9/2}\Delta \xi  +\rho\widehat{\phi}^{-9/2}\Delta v_3 - (\rho\widehat{\phi}^{-9/2})_t\Delta \varphi \bigl) dxdt. \nonumber
\end{align}

 Let  us now analyze each one of the terms in \eqref{QX1}.
\begin{align}
& s^9\int \! \! \! \int_{\omega_0^5\times (0,T)}\theta_5 e^{2s\alpha}\rho \widehat{\phi}^{-9/2} \phi^{18} \xi \Delta (\Delta \psi)dxdt \nonumber \\
& = -s^9\int \! \! \! \int_{\omega_0^5 \times (0,T)}\nabla (\theta_5 e^{2s\alpha}\rho\widehat{\phi}^{-9/2} \phi^{18}) \cdot  \nabla(\Delta \psi) \xi dxdt - s^9\int \! \! \! \int_{\omega_0^5 \times (0,T)}\theta_5 e^{2s\alpha}\rho \widehat{\phi}^{-9/2} \phi^{18} \nabla \xi \cdot  \nabla(\Delta \psi) dxdt \nonumber \\
& \leq C\biggl( s^{19}\int \! \! \! \int_{\omega_0^5 \times (0,T)}e^{2s\alpha}\widehat{\phi}^{28}|\rho|^2|\xi|^2dxdt +s^{17}\int \! \! \! \int_{\omega_0^5 \times (0,T)}e^{2s\alpha}\widehat{\phi}^{26}|\rho|^2|\nabla \xi|^2dxdt  \biggl) +\delta \widehat{I}_0(s,\Delta \psi),
\end{align}
since
$$
| \nabla (\theta_5 e^{2s\alpha}\rho\widehat{\phi}^{-9/2} \phi^{18})| \leq Cs\widehat{\phi}^{29/2} \rho e^{2s\alpha}1_{\omega_0^5}.
$$
Notice that
\begin{align}\label{QX2}
s^{17}\int \! \! \! \int_{\omega_0^6 \times (0,T)}\theta_6e^{2s\alpha}\widehat{\phi}^{26}|\rho|^2|\nabla \xi|^2dxdt  & = -s^{17}\int \! \! \! \int_{\omega_0^6\times (0,T)}\theta_6 e^{2s\alpha}\widehat{\phi}^{26}|\rho|^2\Delta \xi \xi dxdt  \nonumber \\
&+ \frac{s^{17}}{2}\int \! \! \! \int_{\omega_0^6 \times (0,T)}\Delta (\theta_6 e^{2s\alpha}\widehat{\phi}^{26}|\rho|^2)  |\xi|^2 dxdt  \\
& \leq C s^{33}\int \! \! \! \int_{\omega_0^6 \times (0,T)} e^{2s\alpha} \widehat{\phi}^{61}|\rho|^2|\xi|^2 dxdt + \delta I_2(s,\rho \widehat{\phi}^{-9/2} \xi),\nonumber
\end{align}
because
$$
|\Delta (\theta_6 e^{2s\alpha}\widehat{\phi}^{26}|\rho|^2) | \leq Cs^2\widehat{\phi}^{28} |\rho|^2e^{2s\alpha}1_{\omega_0^6}.
$$

Next,
\begin{align}\label{QX3}
& s^9\int \! \! \! \int_{\omega_0^5 \times (0,T)}\theta_5 e^{2s\alpha}\widehat{\phi}^{-9/2} \phi^{18} \rho\xi \widehat{\phi}^{-9/2}\rho\Delta \xi dxdt \nonumber \\
&=Cs^9\int \! \! \! \int_{\omega_0^5 \times (0,T)}\Delta(\theta_5 e^{2s\alpha} \widehat{\phi}^{9} |\rho|^2)  |\xi|^2 dxdt - Cs^9\int \! \! \! \int_{\omega_0^5 \times (0,T)}\theta_5 e^{2s\alpha} \widehat{\phi}^{9} |\rho|^2  |\nabla \xi|^2 dxdt \nonumber \\
& \leq C\bigl(s^{11}\int \! \! \! \int_{\omega_0^5 \times (0,T)}e^{2s\alpha}\widehat{\phi}^{11}|\rho|^2|\xi|^2dxdt + s^9\int \! \! \! \int_{\omega_0^5 \times (0,T)}e^{2s\alpha}\widehat{\phi}^{9}|\rho|^2|\nabla \xi|^2dxdt  \bigl)
\end{align}
because
$$
|\Delta (\theta_5 e^{2s\alpha} \widehat{\phi}^{9}  |\rho|^2) | \leq Cs^2\widehat{\phi}^{11} |\rho|^2e^{2s\alpha}1_{\omega_0^5} \ \mbox{and} \ |\widehat{\phi}^{-1}| \leq CT^{22}.
$$
We estimate the term in $v_3$ as follows
\begin{align}\label{QX4}
s^9\int \! \! \! \int_{\omega_0^5 \times (0,T)}&\theta_5 e^{2s\alpha} \widehat{\phi}^{9}  \rho^2 \xi \Delta v_3 dxdt  = s^9\int \! \! \! \int_{\omega_0^5 \times (0,T)}\theta_5 e^{2s\alpha} \widehat{\phi}^{9}\rho\xi\Delta(z+w) dxdt  \nonumber \\
& \leq  C\bigl(s^{13}\int \! \! \! \int_{\omega_0^5 \times (0,T)}e^{2s\alpha}\widehat{\phi}^{13}|\rho|^2|\xi|^2dxdt + \|\rho F_3\|^2_{L^2(0,T;\bold{V})} \bigl)+ \delta s^5\int \! \! \! \int_Q  e^{2s\alpha}\phi^5|\Delta z_3|^2dxdt,
\end{align}
for any $\delta >0$.
Finally, we have
\begin{align*}
&s^9\int \! \! \! \int_{\omega_0^5 \times (0,T)}\theta_5 e^{2s\alpha}\rho \widehat{\phi}^{-9/2} \phi^{18} \xi (\rho\hat{\phi}^{-9/2})_t \Delta \varphi dxdt \\
& = s^9\int \! \! \! \int_{\omega_0^5\times (0,T)}\theta_5 e^{2s\alpha}\phi^{18} \xi (\rho\widehat{\phi}^{-9/2})_t (\Delta \psi +\Delta \eta)dxdt
\end{align*}
and it is not difficult to see that
\begin{align}\label{M2}
| s^9\int \! \! \! \int_{\omega_0^5 \times (0,T)}\theta_5 e^{2s\alpha}\phi^{18}  (\rho\widehat{\phi}^{-9/2})_t \Delta \psi  \xi  dxdt|  &\leq Cs^{19}\int \! \! \! \int_{\omega_0^5 \times (0,T)}   e^{2s\alpha} \widehat{\phi}^{27}|\rho|^2|\xi|^2dxdt \nonumber \\
& \ \ \ +  \delta \widehat{I}_0(s,\Delta \psi),
\end{align}
since
$$
|(\rho\widehat{\phi}^{-9/2})_t| \leq Cs^{1+1/11}\widehat{\phi}^{-3}\rho.
$$

For the other term in \eqref{M1}, we have
\begin{align}
s^9\int \! \! \! \int_{\omega_0^5\times (0,T)}&\theta_5 (e^{2s\alpha}\rho  \phi^{18}\widehat{\phi}^{-9/2})_t  \Delta \psi \xi   dxdt \nonumber \\
&\leq Cs^{19}\int \! \! \! \int_{\omega_0^5 \times (0,T)}   e^{2s\alpha} \widehat{\phi}^{288/11}|\rho|^2|\xi|^2dxdt+  \delta \widehat{I}_0(s,\Delta \psi)
\end{align}
since
$$
|(e^{2s\alpha}\rho  \phi^{18}\widehat{\phi}^{-9/2})_t | \leq Cs^{1+1/11}\widehat{\phi}^{321/22} e^{2s\alpha}\rho.
$$

Therefore, we have the estimate
\begin{align}\label{M1X}
 |s^9\int \! \! \! \int_{\omega_0^5 \times (0,T)}\theta_5 e^{2s\alpha} \phi^{18}\widehat{\phi}^{-9/2} &\rho\Delta \psi \xi_t dxdt  |  \nonumber\\
 \leq &\  C\bigl(s^{33}\int \! \! \! \int_{\omega_0^6 \times (0,T)} e^{2s\alpha} \widehat{\phi}^{61}|\rho|^2|\xi|^2 dxdt  + \|\rho F_3\|^2_{L^2(0,T;\bold{V})}  \bigl) \nonumber \\
  & + \delta \bigl(I_2(s,\rho \widehat{\phi}^{-9/2} \xi)+  \widehat{I}_0(s,\Delta \psi)+ s^5\int \! \! \! \int_Q  e^{2s\alpha}\phi^5|\Delta z_3|^2dxdt\bigl).
\end{align}

\end{document}